\documentclass[12pt]{amsart}

\usepackage{amssymb}
\usepackage{amsmath}
\usepackage{amscd}
\usepackage{float}
\usepackage{graphicx}
\usepackage{amsfonts}
\usepackage{pb-diagram}
\usepackage{eufrak}

\newtheorem{thm}{Theorem}

\newtheorem{lem}{Lemma}
\newtheorem{prop}{Proposition}
\newtheorem{defn}{Definition}
\newtheorem{rem}{Remark}

\begin{document}

 \title[A Categorical Model for the Virtual Braid Group]{A Categorical Model for the Virtual Braid Group}

\author{Louis H. Kauffman}
\address{ Department of Mathematics, Statistics and
 Computer Science, University of Illinois at Chicago,
 851 South Morgan St., Chicago IL 60607-7045, U.S.A.}
\email{kauffman@math.uic.edu}
\urladdr{http://www.math.uic.edu/$~$kauffman/}

\author{Sofia Lambropoulou}
\address{ Departament of Mathematics,
National Technical University of Athens,
Zografou Campus, GR-157 80 Athens, Greece.}
\email{sofia@math.ntua.gr}
\urladdr{http://www.math.ntua.gr/$~$sofia}

\thanks{This research has been co-financed by the European Union (European Social Fund - ESF) and Greek national funds through the Operational Program ``Education and Lifelong Learning" of the National Strategic Reference Framework (NSRF) - Research Funding Program: THALIS. Moreover, both authors were partially supported by UIC, NTUA and MFO}



\date{}

\begin{abstract}

This paper gives a new interpretation of the virtual braid group in terms
of a strict monoidal category $SC$ that is freely generated by one object
and three morphisms, two of the morphisms corresponding to basic pure
virtual braids and one morphism corresponding to a transposition in the
symmetric group. The key to this approach is to take pure virtual braids as
primary. The generators of the pure virtual braid group are abstract solutions
to the algebraic Yang-Baxter equation. This point of view illuminates representations
of the virtual braid groups and pure virtual braid groups via solutions to the
algebraic Yang-Baxter equation. In this categorical framework, the virtual braid group
is a natural group associated with the structure of algebraic braiding.
We then point out how the category $SC$ is related to categories associated
with quantum algebras and Hopf algebras and with quantum invariants of virtual links.
\end{abstract}
\maketitle


\section{Introduction}
This paper gives a new interpretation of the virtual braid group in terms of a tensor category $SC$ with generating morphisms $\mu_{ij}$ where this symbol denotes an abstract connecting string between
strands $i$ and $j$ in a diagram that otherwise is an identity braid on $n$ strands.  These $\mu_{ij}$
satisfy the algebraic Yang-Baxter equation and they generate, in this interpretation, the pure virtual
braid group.  The other generating morphisms  of this category are elements $v_{i}$ that are
depicted as virtual crossings between strings $i$ and $i +1.$ The generators $v_{i}$ have all the relations for transpositions generating the symmetric group.  An $n$-strand diagram that is a product of
these generators is regarded as a morphism from $[n]$ to $[n]$  where the symbol $[n]$  is regarded as
an ordered  row of $n$ points that constitute the top or the bottom of a diagram involving $n$ strands.
The virtual braid group on $n$ strands is isomorphic to the group of morphisms in the String Category $SC$ from $[n]$ to $[n].$ Given that one studies the algebraic Yang-Baxter equation, it is natural to study the compositions of algebraic braiding operators placed in two out of the $n$ tensor lines and to let the permutation group of the tensor lines act on this algebra as the group generated by the virtual crossings.
This construction is in sharp contrast to the role of the virtual crossings in the original form of the virtual knot theory.
\smallbreak

Figure~\ref{musigma} illustrates most of the issues.
At the top of the figure we have illustrated the pure virtual braid  $ \mu = \sigma v$ on two strands.
The permutation associated with $\mu$ is the identity, as each strand returns to its original position.
The braiding element $\sigma$ has been composed with the virtual crossing $v$, which acts as a permutation of the two strands. With these conventions in place we find that
$\mu$ satisfies the {\it algebraic Yang-Baxter equation}
 $$\mu_{12} \mu_{13} \mu_{23} = \mu_{23} \mu_{13} \mu_{12}$$ and this is equivalent to the statement that $\sigma$ satisfies the {\it braiding relation} $$\sigma_{1}\sigma_{2}\sigma_{1} = \sigma_{2}\sigma_{1}\sigma_{2}.$$ This relationship is well-known and it is fundamental to the construction of representations of the Artin braid group and to the construction of quantum link invariants (see
 \cite{Ohtsuki} for an  account of these matters). In the present paper we will detail this relationship once again, and we shall see that it leads to alternative ways to understand the concept of virtual braiding and to generalizations of the formulation of quantum invariants of knots and links to quantum invariants of virtual knots and links (taken up to rotational equivalence described below).
 \smallbreak

 Here a notational issue
leads to a mathematical concept. View Figure~\ref{musigma} and notice how we have diagrammed
the algebraic Yang-Baxter relation. An element $\mu_{ij}$ is shown as a graphical connection between
vertical lines labeled $i$ and $j$ respectively. The vertical lines represent different factors in a tensor product in the usual interpretation where $\mu \in \mathcal{A} \otimes  \mathcal{A}$ where $ \mathcal{A}$ is an algebra that carries a solution to the algebraic Yang-Baxter equation. We call the graphical edge representing $\mu_{ij}$ a {\it string connection} between the strands $i$ and $j.$
{\it The string connection is a topological model for a logical connection in the mathematics.}
The string going from vertical line $1$ to vertical line $3$ represents $\mu_{13},$ and it has nothing to do with strand $2$ except as in the plane the strand $2$ happens to come between strands
$1$ and  $3.$ This means that in our diagram the graphical edge for $\mu_{13}$ intersects the vertical strand $2.$ This intersection is {\it virtual} in the sense that it is just an artifact of the planar drawing.
There is no conceptual connection between $\mu_{13}$ and strand $2$.
\smallbreak

We see that virtuality in the sense of artifactual coinicidence of topological entities will be a necessity in depicting logical connection as topological connection. For this reason, the string diagrammatics that we have adopted for the
algebraic Yang-Baxter equation can be taken as a starting point for the development of the virtual braid group. In this paper, we have started with the usual virtual braid group and reformulated it in this algebraic context.
The attentive reader will see that one could start with the formalism of the algebraic Yang-Baxter equation, construct the appropriate categories and first arrive at the pure virtual braid group and then at the virtual braid group. All of these constructions come from the concept of making  topological models for logical connections in mathematical structures.
\smallbreak

The Artin braid group $B_{n}$ is motivated by a combination of topological considerations and the desire for a group structure that is very close to the structure of the symmetric group $S_{n}.$
The virtual braid group $VB_{n}$ is motivated at first by a natural extension of
the Artin braid group in the context of virtual knot theory. The virtual crossings appear as artifacts of the presentation of virtual knots in the plane where those knots acquire extra crossings that are not really part of the essential structure of the virtual knot. We add virtual crossings to
the Artin braid group and follow the principles of virtual knot theory for handling them. These virtual crossings appear crucially in the virtual braid group, and turn into the generators of the symmetric group embedded in the virtual braid group.  Thus we arrive at the action of the symmetric group in either case,
but with different motivations. Seen from the categorical view, the virtual crossings are interpreted as generators of the symmetric group whose action is added to the algebraic structure of the pure virtual braid group, and they become part of the embedded symmetry of the structure of the virtual braid group. The pure virtual  braid group is seen to be a natural monoidal category generated by formal elements satisfying the algebraic Yang-Baxter equation. The virtual braid group is then an extension of the pure virtual braid group  by the symmetric group.
It has nothing to do with the plane and nothing to do with virtual crossings. It is a natural group associated with the structure of algebraic braiding. This is our motivation for constructing the category $SC.$
\smallbreak

Here is a quick technical description of our category.
We define a strict monoidal category
 $SC$ that is
freely generated by one object $*$ and three morphisms
$\mu: * \otimes * \longrightarrow * \otimes *,$ $\mu': * \otimes * \longrightarrow * \otimes *,$
and $v:* \otimes * \longrightarrow * \otimes *.$ This basic structure, subjected to appropriate relations
can be understood via morphisms  $\mu_{ij}$ defined in terms of the generating morphisms, where the symbol
 $\mu_{ij}$ can is interpreted as a connection between strands $i$ and $j$ in a diagram that otherwise is an identity  on $n$ strands.  The $\mu_{ij}$ satisfy the algebraic Yang-Baxter equation in the sense that for $i < j< k$,
$\mu_{ij}\mu_{ik}\mu_{jk} = \mu_{jk}\mu_{ik}\mu_{ij}.$
The other basic morphisms of this category are elements $v_{i}$ that can be
depicted as virtual crossings between strings $i$ and $i +1.$ The $v_{i}$ are obtained from $v$ by tensoring with identity morphisms
$* \longrightarrow *.$ The $v_{i}$ generate the symmetric group $S_{n}.$ The $\mu_{ij}$ are obtained from $\mu$ by the action of the symmetric group that is generated by the $v_{i}.$ Composition with an  individual $v_{i}$ makes a transposition of indices on the $\mu_{kl},$ generating all of them from the basic $\mu$ and $\mu'.$ An $n$-strand diagram that is a product of
basic morphisms is a morphism from $[n]$ to $[n]$  where the symbol $[n]$  is
an ordered  row of $n$ points that constitute the top or the bottom of a diagram involving $n$ strands.
Here  $[n] = * \otimes *  \cdots * \otimes *$ for
a tensor product of $n$ $*$'s.  In Figure~\ref{musigma} we illustrate the diagrammatic interpretation of
$\mu$ and the fundamental relation of $\mu$ and $v$ with an elementary braiding element $\sigma.$
The relation is $\mu = \sigma v.$ The virtual braid $\sigma v$ is pure in the sense that its associated
permutation is the identity.
\smallbreak

\begin{figure}
     \begin{center}
     \includegraphics[width=6cm]{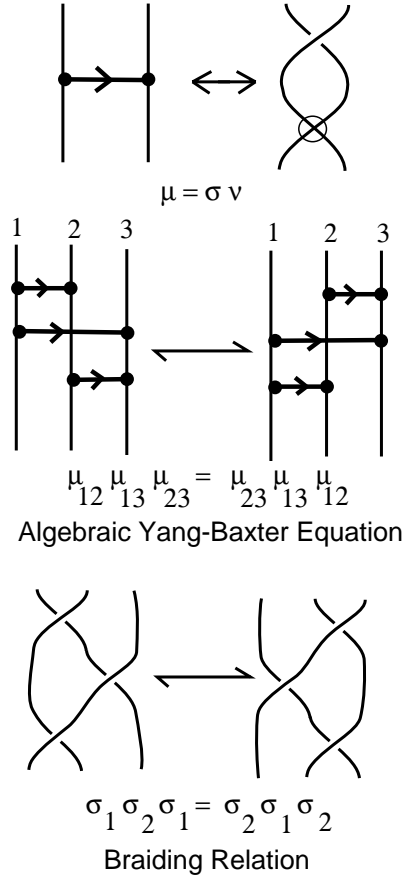}
     \caption{ Algebraic Yang-Baxter Equation and Braiding Relation }
     \label{musigma}
\end{center}
\end{figure}

The category we describe is a natural structure for an algebraist interested in exploring formal properties of the algebraic Yang-Baxter equation, and it is directly related to more topological points of view about virtual links and virtual braids. In fact, a closely related category, under differrent motivation, was constructed in \cite{KRH} where the intent was to construct a category that would be naturally associated with a Hopf algebra on the one hand, and would receive topological tangles, knots and links under a functor from the tangle category to the Hopf algebra category. The present category, giving the structure of the virtual braid group, is a subcategory of that category associated with a general Hopf algebra. We explain this relationship in detail in Section 6 of the present paper. See also Remark 10 of \cite{Bellingeri} and references therein for another earlier observation of the relationship of the algebraic Yang-Baxter equation with the pure virtual braid group.
\smallbreak

We now describe exactly the structure of the paper.
We develop our model for the virtual braid group by first recalling, in Section 2,  its usual definition motivated by virtual knot theory. We then proceed  to reformulate the virtual braid group in terms of the above mentioned generators. By the time we reach Theorem 1, we have reformulated the virtual
braid group in terms of the new generators. We then use this approach to give a presentation of the
pure virtual braid group in Theorem 2.
More precisely, in Section 2 we give a presentation for the virtual braid group in terms of our stringy model. We start by describing the usual presentation of the virtual braid group in terms of classical
braid generators and virtual generators that act as permutations between pairs of adjacent strands in the braid, and relations among them (see Figures~ \ref{genrs} -- \ref{forbiddenpic}). Elementary connecting strings (see Figure~\ref{elstrings}) are defined as elementary pure virtual braids --
products of braid generators and virtual generators as in Figure~\ref{musigma}. We then generalize the notion of connecting
string and show that it has the formal diagrammatic property of being stretched and contracted as shown in Figure~\ref{longstrings}. This property makes the string a topological model for a logical connection as we have advertised earlier in this introduction. With these constructions we then rewrite presentations for the virtual braid group and, in Section 3, show how the connection with strings generates the pure virtual  braid group with a set of relations that  correspond to the algebraic Yang-Baxter equation. See Theorem~2.
\smallbreak

In Section 4 we construct the String Category discussed in this introduction and we show that
the virtual braid group on $n$ strands is isomorphic to the group of morphisms in the String Category $SC$ from $[n]$ to $[n]$ (see Theorem 3).  In Section 5 we detail the relationship with the algebraic Yang-Baxter equation and show how to use solutions of the algebraic Yang-Baxter equation to obtain representations of the pure virtual braid group and virtual braid group.
In Section 6 we discuss a generalization of the virtual
braid group to the virtual tangle category.
We show in this section how our work on the structure of the virtual braid group fits into the structure of the virtual tangle category. The virtual tangle category can be used for obtaining invariants of knots and links via Hopf algebras. The invariants we obtain are invariants of {\em rotational virtual knots and links}
where the term rotational means that we do not allow the use of the first virtual Reidemeister move.
See Figure~\ref{vmoves}.
For the virtual tangle category, the rules for regular isotopy of rotational virtuals are
shown in Figure~\ref{regtang}. This is a most convenient category for working with virtual knots and links, and every quantum link invariant
for classical knots and links extends to an invariant for rotational virtual knots and links.  In this section we show how a generalization
of the string connectors defined previously in the paper enables the construction of quantum virtual link
invariants associated with Hopf algebras. The paper ends with two subsections on Hopf algebras. The concept of a quasi-triangular Hopf algebra creates an algebraic context for
solutions to the algebraic Yang-Baxter equation. This algebraic context gives rise to categories and relationships with knot theory and virtual knot theory that connect directly with the contents of  our investigation.

\section {A Stringy Presentation for the Virtual Braid Group}

\subsection{The virtual braid group}

 Let's begin with a presentation for the virtual braid group.
The set of isotopy classes of virtual braids on $n$ strands forms a group, the {\it virtual braid
group}  denoted $VB_n$, that was introduced in \cite{VKT}.  The group operation is the usual braid
multiplication (form $bb'$ by attaching the bottom strand ends of $b$ to the top strand ends of $b'$).
$VB_n$  is generated by the usual braid generators
$\sigma_{1},\ldots, \sigma_{n-1}$ and by the virtual generators $v_{1},\ldots, v_{n-1},$
where each virtual  crossing $v_{i}$ has the form of the braid generator $\sigma_{i}$  with the
crossing replaced by a virtual crossing. See Figure~\ref{genrs} for illustrations. Recall that in  virtual
crossings we do not distinguish between under and over crossing. Thus, $VB_n$  is an extension of
the classical braid group $B_n$ by  the symmetric group $S_n$, whereby $v_i$ corresponds to the elementary transposition $(i,i+1)$.

\begin{figure}
     \begin{center}
     \includegraphics[width=7cm]{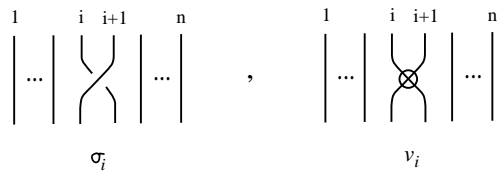}
     \caption{ The generators of $VB_n$ }
     \label{genrs}
\end{center}
\end{figure}

\smallbreak

 Among themselves the braid generators satisfy the usual {\it braiding relations}:
$$ \begin{array}{lcccl}
\text{(B1)} & \sigma_{i}\sigma_{i+1}\sigma_{i} & = & \sigma_{i+1}\sigma_{i} \sigma_{i+1}, &   \\
\text{(B2)} & \sigma_i \sigma_j & = & \sigma_j \sigma_i, & \mbox{for} \ j\neq i\pm 1.  \\
\end{array}$$
 Among themselves, the virtual generators are a presentation for the symmetric group $S_{n}$, so
they satisfy the following {\it virtual relations:}
$$ \begin{array}{lcccl}
\text{(S1)} & v_{i}v_{i+1}v_{i} & = & v_{i+1}v_{i} v_{i+1}, &    \\
\text{(S2)} & v_i v_j & = & v_j v_i, & \mbox{for} \ j\neq i\pm 1,  \\
\text{(S3)} & {v_{i}}^2 & = & 1. &  \\
\end{array}$$
 The {\it mixed relations} between virtual generators and braiding generators are as follows:
$$ \begin{array}{lcccl}
\text{(M1)} & v_{i} \sigma_{i+1} v_{i} & = & v_{i+1} \sigma_i v_{i+1}, &  \\
\text{(M2)} &  \sigma_i v_j & = & v_j \sigma_i, & \mbox{for} \ j\neq i\pm 1.    \\
\end{array}$$
To summarize, the virtual braid group $VB_{n}$ has the following presentation \cite{VKT}.

\begin{equation}\label{brpresi}
VB_{n} = \left< \begin{array}{ll}  \begin{array}{l}
\sigma_1, \ldots ,\sigma_{n-1},  \\
 v_1, \ldots, v_{n-1} \\
\end{array} &
\left|
\begin{array}{l} (B1), (B2),  \\
(S1), (S2), (S3), \\
(M1), (M2)
\end{array} \right.  \end{array} \right>
\end{equation}

 It is worth noting at this point that the virtual braid group $VB_n$ does not embed in the classical braid group $B_n$, since the virtual braid group contains torsion elements (the $v_{i}$ have order two)  and it is well--known that $B_n$ does not. But the classical braid group embeds in the virtual braid group just as classical knots embed in virtual knots.  This fact may be most easily deduced from \cite{KUP}, and can also be seen from \cite{M} and \cite{FRR}. For reference to previous work on virtual knots and braids the reader should consult
 \cite{CS1,DK,GPV,HR,HRK,VKT,SVKT,DVK,KADOKAMI,Kamada,KiSa,KUP,M,Satoh,TURAEV,V1,V2,KL3,KL4} and references therein. For work on welded braids and welded knots, see \cite{FRR,Kamada,KL3,KL4}.
For  Markov--type theorems for virtual braids (and welded braids), giving sets of moves on virtual
braids that generate the same equivalence classes as the oriented virtual link types of their closures, see \cite{Kamada} and \cite{KL4}. Such
theorems are important for understanding the structure and classification of virtual knots and
links.

\smallbreak

 The second mixed relation in the presentation of the virtual braid group will be called the {\it local detour move} and it is illustrated in Figure~\ref{locdetour}. The following  relations are also local detour moves for virtual
braids and they are easy consequences of the above.

\begin{equation}\label{}
\begin{array}{ccc}
v_{i}v_{i+1} {\sigma_i}^{\pm 1} & = & {\sigma_{i+1}}^{\pm 1}v_{i}v_{i+1},   \\
{\sigma_{i}}^{\pm 1} v_{i+1}v_{i} & = & v_{i+1}v_{i} {\sigma_{i+1}}^{\pm 1}.    \\
\end{array}
\end{equation}
 This set of relations taken together define the basic local isotopies for virtual braids. Each relation is a braided version of a local virtual link isotopy. The local detour move is written equivalently:
\begin{equation}\label{sigmadet}
\sigma_{i+1} = v_i v_{i+1} \sigma_i v_{i+1} v_i.
\end{equation}
Notice that Eq.~\ref{sigmadet} is  the braid detour move of the $i$th strand around
the crossing between the $(i+1)$-st and the
$(i+2)$-nd strand (see first two illustrations in Figure~\ref{detsigma}) and it provides an inductive way of  expressing
all  braiding generators in terms of the first braiding generator $\sigma_{1}$  and the virtual
generators $v_1,\ldots,v_{n-1}$ (see first and last illustrations in Figure~\ref{detsigma}), that is:
\begin{equation}\label{}
\sigma_{j} = (v_{j-1}\ldots v_2v_1)\,(v_{j}\ldots v_3v_2)\,\sigma_1\, (v_2v_3\ldots
v_{j})\,(v_1v_2\ldots v_{j-1}).
\end{equation}

\begin{figure}
     \begin{center}
     \includegraphics[width=3.5cm]{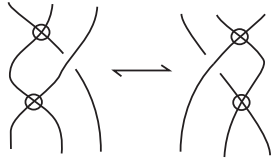}
     \caption{ The local detour}
     \label{locdetour}
\end{center}
\end{figure}

\smallbreak

\begin{figure}
     \begin{center}
     \includegraphics[width=11cm]{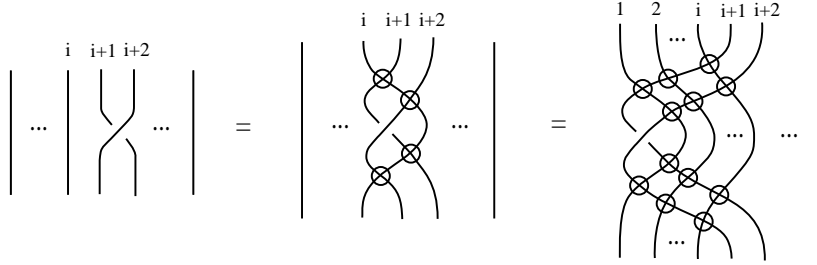}
     \caption{ Detouring the crossing $\sigma_{i+1}$}
     \label{detsigma}
\end{center}
\end{figure}

\noindent In \cite{KL3} we derive the following reduced presentation for
$VB_{n}$:
\begin{equation}\label{brreduced}
VB_{n} = \left< \begin{array}{ll}
\begin{array}{c}
\sigma_1, \\
v_1, \ldots ,v_{n-1}  \\
\end{array} &
\left|
\begin{array}{l}(S1), (S2), (S3) \\
\sigma_1 v_j=v_j \sigma_1, \ \ \  \mbox{for} \  j>2   \\
 v_1 \sigma_1 v_1\,v_2 \sigma_1 v_2\, v_1 \sigma_1 v_1 =
  v_2 \sigma_1 v_2\, v_1 \sigma_1 v_1 \, v_2 \sigma_1 v_2   \\
\sigma_1\, v_2 v_3 v_1 v_2 \sigma_1 v_2 v_1 v_3 v_2 =
    v_2 v_3 v_1 v_2 \sigma_1 v_2 v_1 v_3 v_2 \, \sigma_1   \\
\end{array} \right.  \end{array} \right>
\end{equation}

The local detour move gives rise to a generalized {\it detour move}, by which any box in the braid can be detoured to any position in the braid, see Figure~\ref{boxdet}.

\smallbreak

\begin{figure}
     \begin{center}
     \includegraphics[width=8cm]{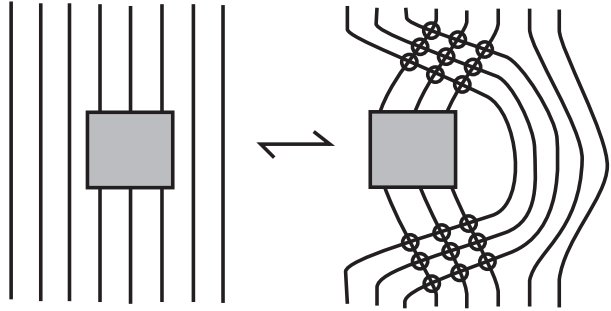}
     \caption{ Detouring a box}
     \label{boxdet}
\end{center}
\end{figure}

Finally, it is worth recalling that in virtual knot theory there are ``forbidden moves" involving two real crossings and one virtual. More precisely, there are two types of
forbidden moves: One with an over arc, denoted $F_1$ and another with an under arc, denoted $F_2$. See \cite{VKT} for explanations and interpretations. Variants of the forbidden moves are illustrated in Figure~\ref{forbiddenpic}. So, relations of the types:

\begin{equation}\label{forbidden}
 \sigma_{i} v_{i+1} \sigma_{i}^{-1} = \sigma_{i+1}^{-1} v_{i} \sigma_{i+1} \quad (F1) \qquad \mbox{and} \qquad \sigma_{i}^{-1} v_{i+1} \sigma_{i} = \sigma_{i+1} v_{i} \sigma_{i+1}^{-1} \quad (F2)
\end{equation}
 are not valid in virtual knot theory.

\begin{figure}
     \begin{center}
     \includegraphics[width=10cm]{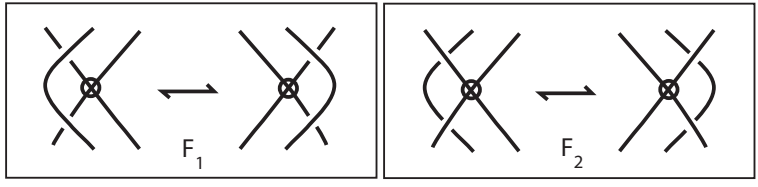}
     \caption{The forbidden moves}
     \label{forbiddenpic}
\end{center}
\end{figure}

\subsection{}

We now wish to describe a new set of generators and relations for the virtual braid group  that makes it particularly easy to describe the pure virtual braid group, $VP_{n}$. In order to accomplish
this aim, we introduce the following elements of $VP_{n}$, for $i=1,\ldots,n-1$.
\begin{equation}\label{mui}
\mu_{i,i+1} := \sigma_{i} v_{i}
\end{equation}
 We indicate $\mu_{i,i+1}$ by a  connecting string between the $i$-th and $(i+1)$-st strands  with a dark vertex on the $i$-th strand, a dark vertex on the $(i+1)$-st strand, and an arrow from left to right. View Figure~\ref{elstrings}. The inverses $\mu_{i,i+1}^{-1} = v_{i}\sigma_{i}^{-1} $
have same directional arrows but are indicated by using white vertices. By detouring it  to the leftmost position of the braid, we can write $\mu_{i,i+1}$ in terms of $\mu_{12}$ conjugated by a virtual word:
\begin{equation}\label{muimu12}
\mu_{i,i+1} = (v_{i-1}\ldots v_2v_1)(v_{i}\ldots v_3v_2) \mu_{12} (v_2 v_3 \ldots v_{i}) (v_1 v_2\ldots  v_{i-1}).
\end{equation}

We also introduce the elements
\begin{equation}\label{muibar}
\mu_{i+1,i} := v_{i} \sigma_{i}  = v_{i} \mu_{i,i+1}v_{i}
\end{equation}
We indicate $\mu_{i+1,i}$ by a  connecting string  between the  $i$-th and $(i+1)$-st strands, with a dark vertex
on the $i$-th strand, a dark vertex on the $(i+1)$-st strand, and an arrow from right to left  (reversing the direction from $\mu_{i,i+1}$), view Figure~\ref{elstrings}. An illustration of Eq.~\ref{muibar} (see top of Figure~\ref{relstrings}) explains the reversing of the direction of the arrow in the graphical interpretation of $\mu_{i+1,i}$. The inverses  $\mu_{i+1,i}^{-1} = \sigma_{i}^{-1} v_{i}$ have same directional arrows but are indicated by using white vertices. An analogous equation to Eq.~\ref{muimu12} holds:
\begin{equation}\label{muimu12bar}
\mu_{i+1,i} = (v_{i-1}\ldots v_2v_1)(v_{i}\ldots v_3v_2) \mu_{21} (v_2 v_3 \ldots v_{i}) (v_1 v_2\ldots  v_{i-1})
\end{equation}

\begin{defn} \rm
The pure virtual braids $\mu_{i,i+1}, \mu_{i+1,i}$ and their inverses shall be called {\it elementary connecting strings}.
\end{defn}

\begin{figure}
     \begin{center}
     \includegraphics[width=6cm]{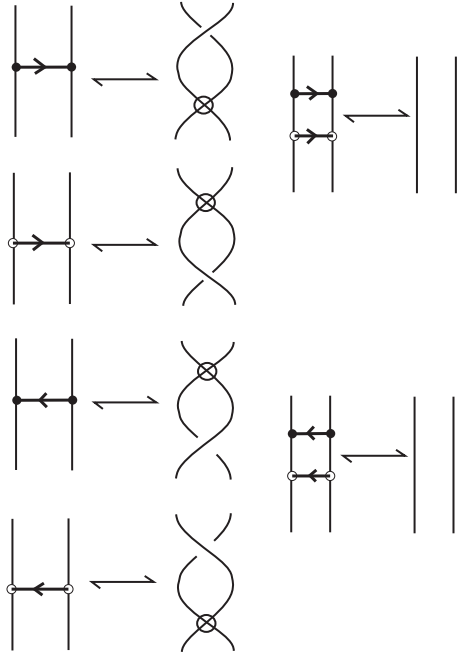}
     \caption{ The elementary connecting strings $\mu_{i,i+1}$, $\mu_{i+1,i}$ and their inverses}
     \label{elstrings}
\end{center}
\end{figure}

\smallbreak

 From Eqs.~\ref{mui} and \ref{muibar} follow directly the relations:
 \begin{equation}\label{muimuibar}
v_i \mu_{i+1,i} = \mu_{i,i+1} v_i \ \ \mbox{and} \ \ \mu_{i+1,i}^{-1} v_i = v_i \mu_{i,i+1}^{-1},
\end{equation}
also illustrated in Figure~\ref{relstrings}.

\begin{figure}
     \begin{center}
     \includegraphics[width=5.2cm]{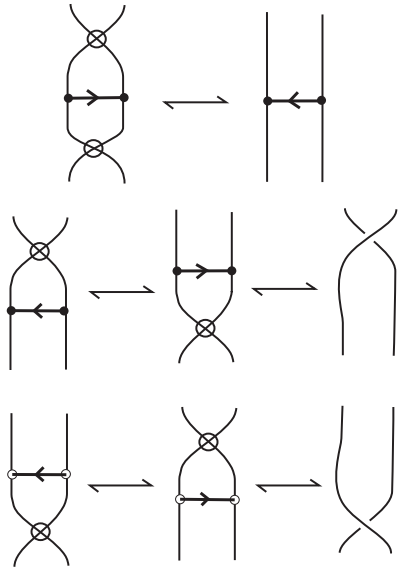}
     \caption{ Relations between the elementary connecting strings}
     \label{relstrings}
\end{center}
\end{figure}

\smallbreak

Further, we generalize the notion of a connecting string and define, for $i<j$, the element $\mu_{ij}$ (a connecting string from strand $i$ to strand $j$) by the formula
\begin{equation}\label{muij}
\mu_{ij} := v_{j-1} v_{j-2} \ldots v_{i+1} \, \mu_{i,i+1} \, v_{i+1} \ldots v_{j-2} v_{j-1}.
\end{equation}
In a diagram $\mu_{ij}$ is denoted by a connecting string from strand $i$ to strand $j$, with dark vertices on these two strands and an arrow pointing from left to right, view Figure~\ref{longstrings}.

\begin{figure}
     \begin{center}
     \includegraphics[width=7cm]{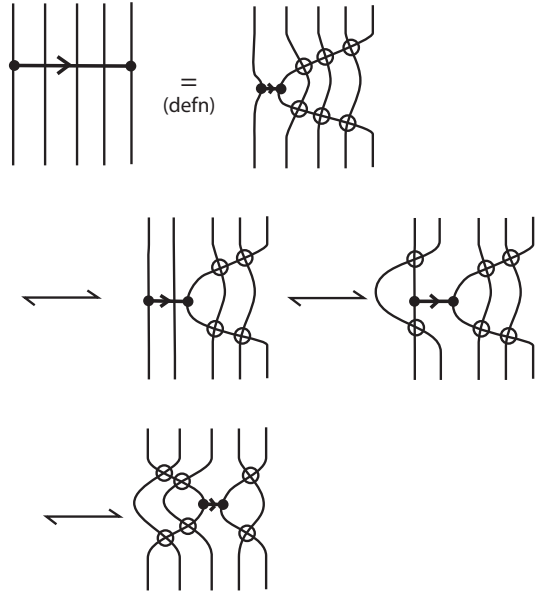}
     \caption{ Connecting strings}
     \label{longstrings}
\end{center}
\end{figure}

\smallbreak

We also generalize, for $i<j$, the elements $\mu_{i+1,i}$ to the elements:
\begin{equation}\label{muijbar}
\mu_{ji} := t_{ij}\, \mu_{ij}\, t_{ij}
\end{equation}
where $t_{ij} = v_{i} v_{i+1} \ldots v_{j} \ldots v_{i+1} v_{i}$ is the element of $S_{n}$ (generated by the $v_{i}$'s) that interchanges strands $i$ and $j$, leaving all other strands fixed.
We denote $\mu_{ji}$ by a connecting string from strand $i$ to strand $j$, with dark vertices, and an arrow pointing from right to left.  Figure~\ref{exchange} illustrates the example $\mu_{31} = v_{2}v_{1}v_{2}\mu_{13}v_{2}v_{1}v_{2}$. It is easily verified that
\begin{equation}\label{muijbar2}
\mu_{ji} = v_{j-1} v_{j-2} \ldots v_{i+1} \, \mu_{i+1,i} \, v_{i+1} \ldots v_{j-2} v_{j-1}
\end{equation}

\smallbreak

\begin{figure}
     \begin{center}
     \includegraphics[width=4.5cm]{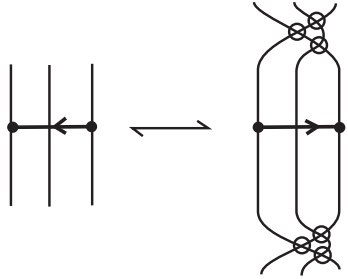}
     \caption{ The exchange of labels between $\mu_{ij}$ and $\mu_{ji}$}
     \label{exchange}
\end{center}
\end{figure}

\noindent The inverses of the elements $\mu_{ij}$ and $\mu_{ji}$ have same directional arrows respectively, but white dotted vertices.

\begin{defn}\rm
The elements $\mu_{ij}, \ \mu_{ji}$ and their inverses shall be called {\it connecting strings}.
\end{defn}

With the above conventions we can speak of connecting strings $\mu_{rs}$ for any $r,s$. It is important to have the elements $\mu_{ji}$ when $j > i$, but in the algebra they are all defined in terms of the $\mu_{ij}$. The importance of  having the elements $\mu_{ji}$ will become clear when we restrict to the pure virtual braid group.

\begin{rem} \rm In the definition of $\mu_{ij}$ if we consider $\mu_{i,i+1}$ as a virtual box inside the virtual braid we can use the (generalized) detour moves to bring it to any position, as Figure~\ref{longstrings} illustrates. This means that the contraction of $\mu_{ij}$ to $\mu_{i,i+1}$ may be pulled anywhere between the $i$-th and the $j$-th strands. By the same reasoning the contraction of $\mu_{ji}$ to $\mu_{i+1,i}$ may be also pulled anywhere between the $i$-th and the $j$-th strands.
\end{rem}

\subsection{}

We shall next give some relations satisfied by the connecting strings. Before that we need the following remark.

\begin{rem}\label{action} \rm
The symmetric group $S_n$ clearly acts on $VB_n$ by conjugation. By their definition (Eqs. \ref{mui}, \ref{muibar}, \ref{muij}, \ref{muijbar}, \ref{muijbar2}), this action on connecting strings translates into permuting their indices, that is, a permutation $\tau \in S_n$ acting on  $\mu_{rs}$ will change it to $\mu_{\tau (r),\tau (s)}$. This means that $S_n$ acts by conjugation also on the subgroup of $VB_n$ generated by the $\mu_{ij}$'s. Moreover, by Eqs.~\ref{muimu12}, \ref{muibar}, all connecting strings may be obtained by the action of $S_n$ on $\mu_{12}$. For $\sigma \in S_{n}$ we regard $\sigma$ both as a product of the elements $v_{i}$ and as a permutation of the set $\{1,2,3, \ldots, n \}.$

Further, any relation in $VB_n$ transforms into a valid relation after acting on it an element of $S_n$. In particular, a commuting relation between connecting strings will be transformed to a new commuting relation between connecting strings.
\end{rem}

\begin{lem}\label{stringyslides}
The following relations hold  in $VB_n$ for all $i$.
\begin{enumerate}
\item $v_{i} \mu_{i,i+1} = \mu_{i+1,i} v_{i}$ \ , \ \ $v_{i} \mu_{i+1,i}  = \mu_{i,i+1} v_{i}$
\item $v_{i+1} \mu_{i,i+1} = \mu_{i,i+2} v_{i+1}$ \ , \ \ $
v_{i+1} \mu_{i+1,i} =  \mu_{i+2,i} v_{i+1}$
\item $v_{i-1} \mu_{i,i+1}  =  \mu_{i-1,i+1} v_{i-1}$ \ , \ \ $
v_{i-1} \mu_{i+1,i} =  \mu_{i+1,i-1} v_{i-1}$
\item $v_{j} \mu_{i,i+1}  =  \mu_{i,i+1} v_{j}$ \ , \ \ $
v_{j} \mu_{i+1,i}  =   \mu_{i+1,i} v_{j}, \ \ j\neq i-1,\, i,\, i+1$.\\
\end{enumerate}
The above local relations generalize to similar ones involving different indices.  Relations 1 are generalized by Eq.~\ref{muijbar}, reflecting the mutual reversing of $\mu_{ij}$ and $\mu_{ji}$, recall Figures~\ref{relstrings} and \ref{exchange}. Relations 2 and 3 are the {\rm local slide moves}, as illustrated in Figure~\ref{slides}, and they generalize to the {\rm slide moves} coming from the defining equations: $\mu_{i+1,k}=v_i \mu_{ik}v_i$ for any $k<i$ or $k>i+1$. Relations 4 and their generalizations: $v_{j} \mu_{ik}  =  \mu_{ik} v_{j}$ for any $k\neq i$ and $j\neq i-1,\, i,\, k-1,\, k$, are all commuting relations. All these relations result from the action of any $\tau \in S_n$ on $\mu_{12}$:
\begin{equation}\label{mixed}
\text{{\boldmath $\tau^{-1} \, \mu_{12} \, \tau  =\mu_{\tau (1),\tau (2)}$} }.
\end{equation}
\end{lem}

\begin{figure}
     \begin{center}
     \includegraphics[width=8cm]{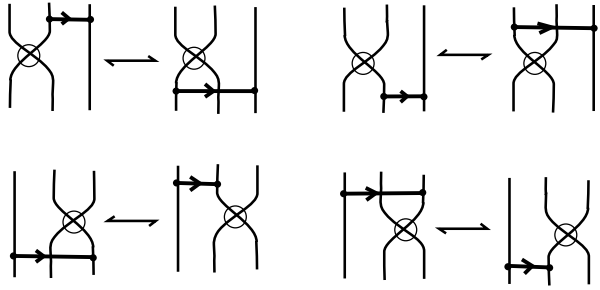}
     \caption{ Slide moves}
     \label{slides}
\end{center}
\end{figure}

\begin{proof}
All relations 1,2 and 3 follow directly from the definitions of the elements $\mu_{ij}$ and
$\mu_{ji}$. For example,  $v_{i+1} \mu_{i,i+1} = \mu_{i,i+2} v_{i+1}$ is equivalent to the defining relation $\mu_{i,i+2} = v_{i+1} \mu_{i,i+1} v_{i+1}$. Figure~\ref{slidepf} illustrates the proof of a local slide move. Relations 4 follow immediately from the commuting relations (S2) and (M2) of $VB_n$. The generalizations of all types of moves follow from the local ones after using detour moves. Finally, the derivation of all relations from the action of $S_n$ on $\mu_{12}$ is explained in Remark~\ref{action} and, more precisely, by the Eqs.~\ref{muimu12}, \ref{muij}, \ref{muimu12bar}, \ref{muijbar2}.
\end{proof}

\begin{figure}
     \begin{center}
     \includegraphics[width=4.5cm]{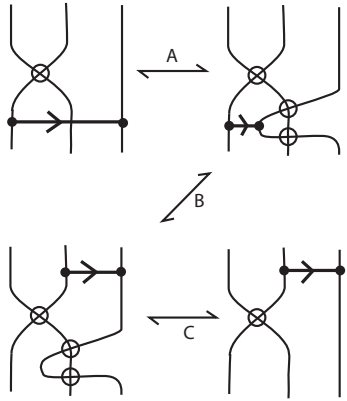}
     \caption{ Proving a local slide move}
     \label{slidepf}
\end{center}
\end{figure}

\begin{lem}\label{stringycom}
The following commuting relations among connecting strings hold in $VB_n$.
\begin{enumerate}
\item $ \text{{\boldmath $\mu_{12} \mu_{34} = \mu_{34} \mu_{12}$}}$
\item $\mu_{14} \mu_{23} = \mu_{23} \mu_{14} \ \ \ \ \ (\text{action by } (324))$
\item $\mu_{13} \mu_{24} = \mu_{24} \mu_{13}  \ \ \ \ \ (\text{action by } (23))$
\end{enumerate}
The above local relations generalize to commuting relations of the form:
\begin{equation}\label{com}
\mu_{ij}\mu_{kl} = \mu_{kl}\mu_{ij}, \ \ \{i,j\} \cap \{k,l\} = \emptyset.
\end{equation}
All the above commuting relations result from relation~1 by  actions of permutations  (indicated for relations~2, 3 to the right of each relation). Moreover, for any choice of four strands there are exactly $24$ such commuting relations that preserve the four strands.
\end{lem}

\begin{proof}
Relation 1 clearly rest on the virtual braid commuting relations (B2) and (M2).  We shall show how relation 2 reduces  to relation 1. In the proof we underline in each step the generators of $VB_n$ on which virtual braid relations are applied.

\vspace{.1in}

\noindent $\begin{array}{rcl}
\mu_{i,i+3} \mu_{i+1,i+2} &  = & v_{i+2} v_{i+1} \mu_{i,i+1} \underline{v_{i+1} v_{i+2} \mu_{i+1,i+2}}  \\

 \ & \stackrel{detour}{=} &  v_{i+2} v_{i+1} \underline{\mu_{i,i+1} \mu_{i+2,i+3}} v_{i+1} v_{i+2}  \\

 \ &  \stackrel{(1)}{=} & \underline{v_{i+2} v_{i+1} \mu_{i+2,i+3}} \mu_{i,i+1} v_{i+1} v_{i+2} \\

 \ & \stackrel{detour}{=} & \mu_{i+1,i+2} v_{i+2} v_{i+1} \mu_{i,i+1} v_{i+1} v_{i+2} \\

 \ & = & \mu_{i+1,i+2} \mu_{i,i+3}.      \\
\end{array}$

\vspace{.1in}

\noindent Figure~\ref{commute} illustrates how relation 3 also reduces  to relation 1. Notice now that relations 2 and 3 can be derived from relation 1 by conjugation by the permutations $(324)$ and $(23)$ respectively. Let us see how this works specifically for relation~2: the indices of relation~1 against the indices of relation~2 induce the permutation $(324) = v_2v_3$. This means that  conjugating relation~1 by the word $v_2v_3$ will yield relation~2.

Notice also that there are $24$ commuting relations in total involving the strands $1,2,3,4$ and indices in any order. Likewise for any choice of four strands. The derivation of all relations from the action of $S_n$ on relation~1 is clear from Remark~\ref{action}.
\end{proof}

\begin{figure}
     \begin{center}
     \includegraphics[width=6cm]{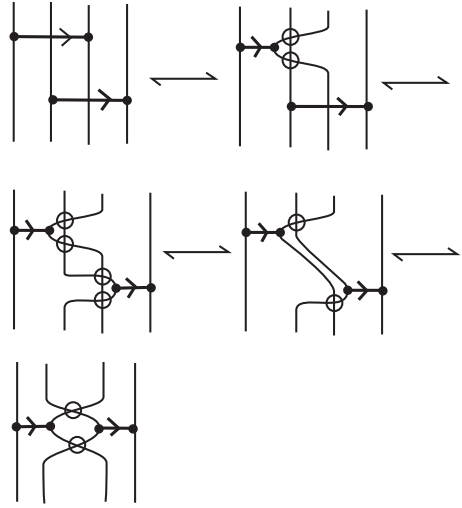}
     \caption{ A local commuting relation}
     \label{commute}
\end{center}
\end{figure}

\begin{lem}\label{stringybr}
The following {\rm stringy braid relations} hold  in $VB_n$.
\begin{enumerate}
\item $\text{{\boldmath $\mu_{12} \mu_{13} \mu_{23} = \mu_{23} \mu_{13}  \mu_{12}$}}$
\item $\mu_{21} \mu_{23} \mu_{13} = \mu_{13} \mu_{23} \mu_{21} \ \ \ \ \ (\text{action by } (12))$
\item $\mu_{13} \mu_{12} \mu_{32} = \mu_{32} \mu_{12} \mu_{13} \ \ \ \ \ (\text{action by } (23))$
\item $\mu_{32} \mu_{31} \mu_{21} =\mu_{21}  \mu_{31} \mu_{32} \ \ \ \ \ (\text{action by } (13))$
\item $\mu_{23} \mu_{21} \mu_{31} = \mu_{31} \mu_{21} \mu_{23} \ \ \ \ \ (\text{action by } (123))$
\item $\mu_{31} \mu_{32} \mu_{12} = \mu_{12} \mu_{32} \mu_{31} \ \ \ \ \ (\text{action by } (132))$
\end{enumerate}
The above relations generalize to three-term relations of the form:
\begin{equation}\label{strbr}
\mu_{ij}\mu_{ik}\mu_{jk} = \mu_{jk}\mu_{ik}\mu_{ij}, \ \ \ \text{for any distinct } i,j,k.
\end{equation}
 All six relations stated above result from the action on relation 1 by permutations of $S_n$, which only permute the indices $\{1,2,3 \}$. These permutations are indicated to the right of each relation. Moreover, for any choice of three strands there are exactly six relations analogous to the above, which all result from relation~1 by actions of appropriate permutations that preserve the three  strands each time.
\end{lem}

\begin{figure}
     \begin{center}
     \includegraphics[width=5cm]{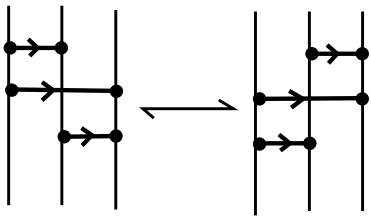}
     \caption{ The stringy braid relation}
     \label{sbr}
\end{center}
\end{figure}

\begin{proof}
Figure~\ref{sbr} illustrates relation 1. Relation~1 rests on the braid relations (B1) of $VB_n$. Indeed, let us prove one relation of this type.  See also Figure~\ref{sybpf} for a pictorial proof.

\vspace{.1in}

\noindent $\begin{array}{rcl}
\mu_{i+1,i+2} \mu_{i,i+2}  \mu_{i,i+1}  & =  & (\sigma_{i+1} \underline{v_{i+1}) (v_{i+1}} \sigma_{i} \underline{v_{i}  v_{i+1}) (\sigma_{i}} v_{i})  \\

 \ & \stackrel{(S3,M1)}{=} & \underline{\sigma_{i+1} \sigma_{i} \sigma_{i+1}} \underline{v_{i}  v_{i+1} v_{i}}    \\

 \ &  \stackrel{(B1,S1)}{=} & \sigma_{i} \sigma_{i+1} \underline{\sigma_{i} v_{i+1} v_{i}} v_{i+1} \\

 \ & \stackrel{(M1,S3)}{=} & \sigma_{i} \sigma_{i+1} \underline{v_{i+1} v_{i} v_{i+1}} v_{i+1} \sigma_{i+1} v_{i+1} \\

 \ & \stackrel{(S1)}{=}& \sigma_{i} \underline{\sigma_{i+1} v_{i} v_{i+1}} v_{i} v_{i+1} \sigma_{i+1} v_{i+1}       \\

 \ & \stackrel{(M1)}{=} &  (\sigma_{i} v_{i}) (v_{i+1} \sigma_{i} v_{i} v_{i+1}) (\sigma_{i+1} v_{i+1})      \\

 \ & = &   \mu_{i,i+1} \mu_{i,i+2} \mu_{i+1,i+2}.     \\

\end{array}$

\vspace{.1in}

The other five stated relations follow from relation~1. Indeed, substituting the $\mu_{ji}$'s from Eqs.~\ref{muibar} and \ref{muijbar}, and drawing the two sides of a relation we notice that there is always a region where, by the slide relations, all three connecting strings become consecutive without any of them having to be reversed, thus enabling application of the first relation. This diagrammatic argument confirms the fact that all six relations are derived from the first one by the action of appropriate elements of $S_n$.  Let us see how this works specifically for relation~5: the indices of relation~1 against the indices of relation~5 induce the permutation $(123) = v_2v_1$. This means that  conjugating relation~1 by the word $v_2v_1$ will yield relation~5.
 Finally, the derivation of all stringy braid relations from the action of $S_n$ on relation~1 is clear from Remark~\ref{action}.
\end{proof}

\begin{figure}
     \begin{center}
     \includegraphics[width=8cm]{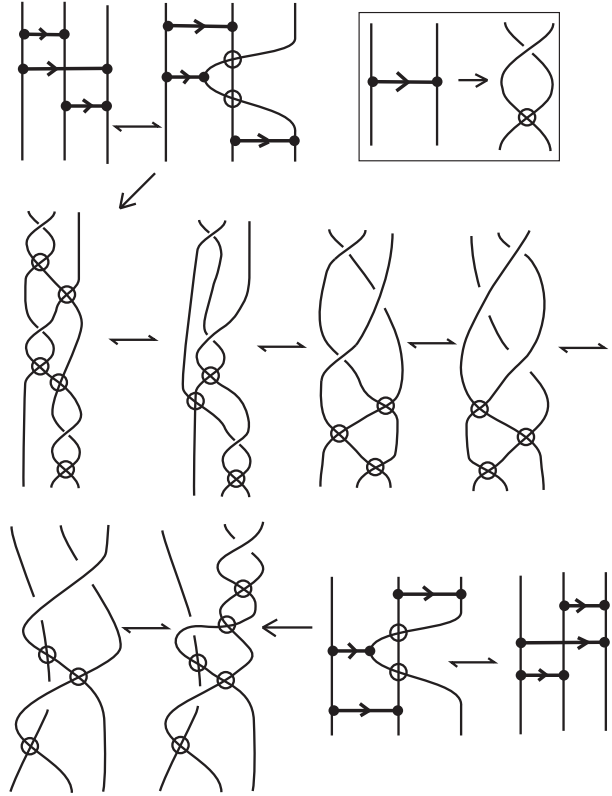}
     \caption{ Proof of the stringy braid relation}
     \label{sybpf}
\end{center}
\end{figure}

Another remark is now due.
\begin{rem}\rm The forbidden moves of virtual knot theory are naturally forbidden also in the stringy category. For example, the forbidden relations $F1, F2$ of Eq.~\ref{forbidden} translate into the following corresponding  {\it forbidden stringy relations} $SF1, SF2$:
\begin{equation}\label{strforbidden}
\mu_{i,i+2} \mu_{i+1,i+2} = \mu_{i+1,i+2} \mu_{i,i+2} \quad (SF1) \qquad \mbox{and} \qquad   \mu_{i,i+2} \mu_{i,i+1} = \mu_{i,i+1} \mu_{i,i+2}  \quad (SF2)
\end{equation}
which, together with all similar relations arising from conjugating the above by permutations, are not valid in the stringy category. See Figure~\ref{sforbidden} for illustrations.

\begin{figure}
     \begin{center}
     \includegraphics[width=8.5cm]{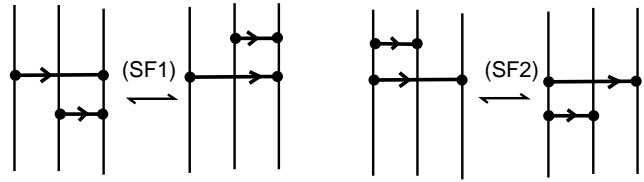}
     \caption{ Stringy forbidden moves}
     \label{sforbidden}
\end{center}
\end{figure}

\end{rem}

\subsection{The stringy presentation}

We will now define an abstract stringy presentation for $VB_n$ that starts from the concept of   connecting string  and recaptures the virtual braid group.
By Eq.~\ref{mui} we have
\begin{equation}\label{sigmai}
\sigma_{i} = \mu_{i,i+1}v_{i}
\end{equation}
so, the connecting strings $\mu_{ij}$ can be taken as an alternate set of generators of the virtual braid group,
along with the virtual generators $v_{i}$. The relations in this new presentation consist in the results  we proved above in Lemmas~\ref{stringyslides}, \ref{stringycom}, \ref{stringybr} describing the
interaction of connecting strings with virtual crossings, the commutation properties of connecting strings, the stringy braiding relations, and the usual relations $(S1), (S2), (S3)$ in the symmetric group $S_{n}$. For the work
below, recall that we have defined the element $t_{ij} = v_{i} v_{i+1} \ldots v_{j} \ldots v_{i+1} v_{i}$
that corresponds to the transposition $(ij)$ in $S_{n}$.

In any presentation of a group $G$ containing the elements $\{ v_1, \ldots, v_{n-1}\}$ and the relations $(S1), (S2), (S3)$ among them,
we have an action of the symmetric group $S_{n}$ on the group $G$ defined by conjugation by an  element $\tau$ in $S_{n}$, expressed in terms of the $v_{i}$:
\[
g^{\tau} = \tau g \tau^{-1}
\]
for $g$ in $G$.  In particular, we can consider $t_{ij}\, g \,t_{ij}$
as the action by the transposition $t_{ij}$ on an element $g$ of $G$.
We will use this action to define a stringy model of the virtual braid group.

\begin{defn}\rm  Let $VS_{n}$ denote the following stringy group presentation.
\begin{equation}\label{stringypres}
VS_{n} =  \left< \begin{array}{ll}  \begin{array}{l}
 \mu_{ij} ,\quad 1 \le i \ne j \le n,  \\
 v_1, \ldots, v_{n-1} \\
\end{array} &
\left|
\begin{array}{l}
\tau \mu_{ij}  \tau^{-1} = \mu_{\tau (i), \tau (j)}, \quad  \tau \in S_{n} \\
 \mu_{12}\mu_{13}\mu_{23} = \mu_{23}\mu_{13}\mu_{12} \\
\mu_{12}\mu_{34} = \mu_{34}\mu_{12}  \\
 (S1), (S2), (S3) \\
\end{array} \right.  \end{array} \right>
\end{equation}
\end{defn}

We can now state the following theorem.

\begin{thm} The stringy group $VS_{n}$ is isomorphic to the virtual braid group
$VB_{n}.$
\end{thm}

\begin{proof} First we define a homomorphism $F: VB_{n} \longrightarrow VS_{n}$ by
$F(v_{i}) = v_{i}$ and $F(\sigma_{i}) = \mu_{i,i+1} v_{i}$, and extend the map to be a
homomorphism on words in the generators of the virtual braid group.  In order to show that this map is well-defined, we must show that it preserves the relations in the virtual braid group.
Since $F(v_{i}) = v_{i}$, the relations among the $v_{i}$ with themselves are preserved
identically. The commuting relations in the braid group are
$\sigma_{i} \sigma_{j} = \sigma_{j} \sigma_{i}$ when $|i-j| > 2$. Thus we must show that
$$
\mu_{i,i+1}v_{i} \mu_{j,j+1}v_{j} =  \mu_{j,j+1}v_{j}\mu_{i,i+1}v_{i}.
$$
But this follows immediately from relations~4 of Lemma~\ref{stringyslides} and from Lemma~\ref{stringycom}. The mixed commuting relations $(M2)$ follow also directly from relations $(S2)$ and relations~4 of Lemma~\ref{stringyslides}. This completes the verification that the commuting relations in the virtual braid group are compatible with $F$.

The detour moves $(M2)$ in the virtual braid group go under $F$ to the slide relations of Lemma~\ref{stringyslides}. We illustrate this in Figure~\ref{sdetour}.

It remains to prove that the braiding relations $(B1)$
carry over to $VS_{n}$ under $F$. Indeed:
$$F(\sigma_{i}\sigma_{i+1}\sigma_{i}) = \mu_{i,i+1}v_{i}\mu_{i+1,i+2}v_{i+1}\mu_{i,i+1}v_{i}$$
$$\stackrel{Lemma~\ref{stringyslides}}{=} \mu_{i,i+1}\mu_{i,i+2}\mu_{i+1,i+2}v_{i}v_{i+1}v_{i},$$
while
$$F(\sigma_{i+1}\sigma_{i} \sigma_{i+1}) = \mu_{i+1,i+2}v_{i+1}\mu_{i,i+1}v_{i}\mu_{i+1,i+2}v_{i+1}$$
$$\stackrel{Lemma~\ref{stringyslides}}{=} \mu_{i+1,i+2}\mu_{i,i+2}\mu_{i,i+1}v_{i+1}v_{i}v_{i+1}.$$
and the two expressions are equal from Lemma~\ref{stringybr} and relations~$(S1)$. This completes the proof that the mapping $F$ is a well-defined homomorphism of groups.

We now define an inverse mapping $G:VS_{n} \longrightarrow VB_{n}$ by
$G(v_{i}) = v_{i}$ and $G(\mu_{i, i+1}) = \sigma_{i} v_{i}$. At this stage we have two pieces of work to accomplish: We must extend $G$ to all of $VB_{n}$ and we must show that $G$ is well-defined and that it preserves the relations in the group presentation. This will be done in the next paragraphs.

First of all, we have the $VS_{n}$ relations:
$$ \tau^{-1}\, \mu_{ij} \, \tau = \mu_{\tau (i), \tau (j)}$$
for all $\tau$ in $S_{n}$. In particular, this means that if $\tau (1) = i$ and
$\tau (2) = j$, then
$$\mu_{ij} = \tau^{-1}\, \mu_{12} \, \tau.$$
Thus we can define
$$G(\mu_{ij}) =  \tau^{-1} G(\mu_{12}) \tau = \tau^{-1} \sigma_{1} v_{1} \tau.$$
It is easy to see that this is well-defined by noting that if $\lambda$ is another
permutation such that $\lambda(1) = i$ and $\lambda(2) = j$, then $\lambda = \tau \gamma$
where $\gamma$ is a permutation that fixes $1$ and $2$. But such a permutation commutes
with $\sigma_{1} v_{1}$ as is easy to see in the virtual braid group. Hence $\lambda$ can replace
$\tau$ in the formula for $G(\mu_{ij})$ with no change. We leave it as an exercise for the reader
to check that our definition of $G(\mu_{i,i+1})$ in the previous paragraph agrees with the present definition. This completes the definition of the map $G$. We now need to see that it respects the
other relations in $VB_{n}.$

We must show that
$$G(\mu_{12}\mu_{34}) = G(\mu_{34}\mu_{12}).$$
Just note that
$$ G(\mu_{12}\mu_{34}) = \sigma_{1}v_{1} \sigma_{3}v_{3} = \sigma_{3 }v_{3} \sigma_{1}v_{1} = G(\mu_{34}\mu_{12}),$$
by the commuting relations in the virtual braid group.

Finally, we must prove
$$G( \mu_{12}\mu_{13}\mu_{23}) = G(\mu_{23}\mu_{13}\mu_{12}).$$
Note that $\mu_{13} = v_{2}\mu_{12}v_{2}$, so we must prove that in the
virtual braid group,
$$\sigma_{1}v_{1} v_{2} \sigma_{1} v_{1} v_{2} \sigma_{2} v_{2} = \sigma_{2}v_{2} v_{2} \sigma_{1} v_{1} v_{2} \sigma_{2} v_{2}.$$
Figure~\ref{sybpf} illustrates how this identity follows via braiding and detour moves.

We have verified that the mapping $G$ is well-defined and, by definition,
the compositions $F\circ G$ and $G\circ F$ are the identity on $VS_{n}$ and
$VB_{n}$. Therefore $VS_{n}$ and $VB_{n}$ are isomorphic groups. This completes
the proof of the Theorem.
\end{proof}

\begin{figure}
     \begin{center}
     \includegraphics[width=4.0cm]{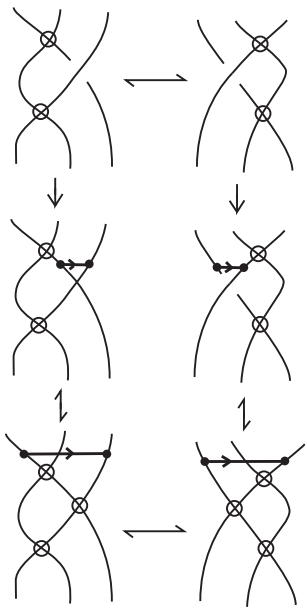}
     \caption{ The detour moves correspond to the slide moves in the stringy category }
     \label{sdetour}
\end{center}
\end{figure}

Finally, we also give below a reduced presentation for $VB_n$, which derives immediately from (\ref{brreduced}).
\begin{prop}
The following is a reduced stringy presentation for $VB_{n}$:
\begin{equation}\label{strreduced}
VB_{n} = \left< \begin{array}{ll}
\begin{array}{c}
\mu_{12}, \\
v_1, \ldots ,v_{n-1}  \\
\end{array} &
\left|
\begin{array}{l}
\mu_{12} v_j=v_j \mu_{12}, \ \ \  \mbox{for} \  j>2   \\
 \mu_{12} \, v_2 \mu_{12} v_2 \, v_1v_2 \mu_{12} v_2 v_1 =
 v_1v_2 \mu_{12} v_2 v_1 \, v_2 \mu_{12} v_2  \, \mu_{12}  \\
 \mu_{12}\, v_2 v_3 v_1 v_2  \mu_{12} v_2 v_1 v_3 v_2 =
   v_2 v_3 v_1 v_2  \mu_{12} v_2 v_1 v_3 v_2  \, \mu_{12}   \\
   (S1), (S2), (S3) \\
\end{array} \right.  \end{array} \right>
\end{equation}
\end{prop}
\noindent The second relation is the stringy braid relation~1 of Lemma~\ref{stringybr} and the third relation is the commuting relation~1 of Lemma~\ref{stringycom}.

\section {The Pure Virtual Braid Group}

\subsection{ A presentation for the pure virtual braid group}

 From presentation Eq.~\ref{brpresi} of $VB_{n}$ we have a surjective homomorphism $$\pi : VB_{n} \longrightarrow S_{n}$$ defined by
$$\pi(\sigma_{i}) = \pi(v_{i}) =v_{i}.$$
For a virtual braid $b$, we refer to $\pi(b)$ as the {\it permutation associated with the virtual braid $b$}, and we define the {\it pure virtual braid group} $VP_{n}$ to be the kernel of the homomorphism $\pi$. Hence, $VP_{n}$ is a normal subgroup of $VB_{n}$ of index $n!$. So, $VP_{n} \cdot S_n = VB_{n}$. Moreover, $VP_n \cap S_n = \{ id \}$. Hence, $VB_{n} = VP_{n} \rtimes S_n$. Equivalently, we have the exact sequence
\[ 1 \longrightarrow VP_{n} \longrightarrow VB_{n} \longrightarrow S_n \longrightarrow 1.\]
A presentation for $VP_n$ can be now derived immediately from the stringy presentation of $VB_{n}$ as an application of the Reidemeister-Schreier process \cite{Fox,MKS,Stillwell}. To see this, we first need the following.

\begin{lem} The subgroup $VP_{n}$ of \ $VB_{n}$ is generated by the elements $\mu_{ij}$ for all $i \neq j$.
\end{lem}

\begin{proof} Indeed, by Eqs.~\ref{mui} and \ref{muibar}, $\sigma_i = \mu_{i,i+1}v_{i} = v_{i} \mu_{i+1,i}$. So,  any element $b \in VB_{n}$ can be written as a product in the $\mu_{ij}$'s and the $v_{k}$'s. Furthermore, by the slide relations of Lemma~\ref{stringyslides}, all $\mu_{ij}$'s can pass to the top of the braid, leaving at the bottom a word $\tau$ in the $v_{k}$'s, such that $\tau = \pi(b)$. Thus, if $b \in VP_{n}$ then $\tau$ must be the  identity permutation. This completes the proof of the Lemma.
\end{proof}

We can now give a stringy presentation of $VP_{n}$.

\begin{thm}  The following is a presentation for the pure virtual braid group.
\begin{equation}\label{purepres}
VP_{n} = \left< \begin{array}{ll}
\begin{array}{c}
\mu_{rs}, \quad  r\neq s \\
\end{array} &
\left|
\begin{array}{l}
\mu_{ij}\mu_{ik}\mu_{jk} = \mu_{jk}\mu_{ik}\mu_{ij}, \quad \text{for all distinct } i,j,k  \\
 \mu_{ij}\mu_{kl} = \mu_{kl}\mu_{ij}, \quad \{i,j\} \cap \{k,l\} = \emptyset  \\
\end{array} \right.  \end{array} \right>
\end{equation}
\end{thm}

\begin{proof} Having reformulated the presentation of the virtual braid group, the proof is now a direct application of the
Reidemeister-Schreier technique.  The relations in $VP_{n}$ arise  as conjugations of the relations in $VB_{n}$ by coset representatives of $VP_{n}$ in $VB_{n}$, which are the elements of $S_n$.    The relations $(S1), (S2), (S3)$ describe $S_{n}$ and are used for choosing the coset representatives. We now describe the process from the point of view of covering spaces.
We have $VP_{n} \subset VB_{n}$ as a normal subgroup with the subgroup $S_{n}$ acting
on it by conjugation. $VP_{n}$ is the fundamental group of the covering space $E$ of a cell complex $B$ with fundamental group $VB_{n}$, where $E$ has
group of deck transformations $S_{n}.$ Since the elements of the symmetric group lift to paths in the covering space, the relations  $\tau \mu_{ij}  \tau^{-1} = \mu_{\tau (i), \tau (j)}$  serve to describe the action of the symmetric group on the loops in the covering space (these loops are the lifts of the elements $\mu_{ij}$).
We choose basic  relations in $VP_{n}$ to be  the lifts at a specific basepoint  of  the braiding relation
 $ \mu_{12}\mu_{13}\mu_{23} = \mu_{23}\mu_{13}\mu_{12}$
 and the commuting relation $\mu_{12}\mu_{34} = \mu_{34}\mu_{12}$. All other relations are obtained from these by the action of $S_{n}$, and  all relations constitute  the two orbits of the basic relations under this action. For example  the
relations
$$\mu_{ij}\mu_{ik}\mu_{jk} = \mu_{jk}\mu_{ik}\mu_{ij}$$
constitute the orbit under the action of $S_{n}$ on the single basic braiding relation
$$\mu_{12}\mu_{13}\mu_{23} = \mu_{23}\mu_{13}\mu_{12}.$$
The same pattern applies to the commuting relations.
This gives the statement of the Theorem and completes the proof.
\end{proof}

\subsection {Semi-Direct Product Structure}

The virtual braid group and the pure virtual braid group can be described in terms of semi-direct products of groups, just as is begun
in the paper by Bardakov \cite{Bardakov} and continued in \cite{Paris}.  In this section we remark that these decompositions are based on the following algebra:
The Yang-Baxter relation has the generic form $$\mu_{i,i+1}\mu_{i,i+2}\mu_{i+1,i+2} = \mu_{i+1,i+2}\mu_{i,i+2}\mu_{i,+1}$$
which is abstractly in the form $$ABC = CBA$$ and can be rewritten in the form $B^{-1}ABC = B^{-1}CBA$ or
$$A^{B} = C^{B}AC^{-1}.$$ This allows one to rewrite some of the Yang-Baxter relations in terms of the conjugation action of the group
on itself, and this is the key to the structural work pioneered by Bardakov.

\section{A String  Category for the Virtual Braid Group}

In this section we summarize our results by pointing out that the string connectors and the virtual crossings can be regarded as generators of a category whose algebraic structure yields the virtual
braid group and the pure virtual braid group. There are many relations in the definition of
this category. These relations all act to make the string connection a topological model of a logical connection between strands in this tensor category. The specific topological interpretations of all these relations have been discussed in the preceding sections of this paper.
\smallbreak

We define a strict monoidal category with generating morphisms
$\mu_{ij}$ where this symbol is interpreted as  an abstract string or connection between
strands $i$ and $j$ in a diagram that otherwise is an identity braid on $n$ strands  just as defined  in the previous sections. The other generators of this category are morphisms $v_{i}$ that are
interpreted as virtual crossings between strings $i$ and $i +1.$ The generators $v_{i}$ have all the relations for transpositions generating the symmetric group. Compositions of these elements generate
the morphisms of the category. The relations among these morphisms are exactly the relations
described for the $v_{k}$ and the $\mu_{ij}$ in the previous sections. We will now define this category
using a minimal number of generators.

\begin{defn} \label{sc}  \rm
Consider the strict monoidal category freely generated by one object $*$ and three morphisms
$$\mu: * \otimes * \longrightarrow * \otimes *,$$ $$\mu': * \otimes * \longrightarrow * \otimes *,$$
and $$v:* \otimes * \longrightarrow * \otimes *.$$ Let $\mu_{12} = \mu \otimes id_{*},$
$\mu_{21} = \mu' \otimes id_{*},$
$v_{1} = v \otimes id_{*}, v_{2} = id_{*} \otimes v.$ Here we express these elements in
three strands (tensor factors). For an arbitrary number of tensor factors, we write
$$v_{i} = id_{*} \otimes  \cdots  \otimes id_{*} \otimes  v  \otimes  id_{*}  \otimes  \cdots  \otimes id_{*}$$ where $v$ occurs in the $i$-th place in this tensor product. More generally, it is understood that
 $$\mu_{12} = \mu \otimes id_{*} \otimes  \cdots \otimes id_{*}$$
 and that $$\mu_{21} = \mu' \otimes id_{*} \otimes  \cdots \otimes id_{*}$$
 for an arbitrary number of tensor factors.
 \smallbreak

\noindent For each natural number $n$, the symbols
$$[n]=  * \otimes * \otimes \cdots \otimes *$$ with $n$ $*$'s  are the objects in the category. One can regard $[n]$ as an ordered  row of $n$ points that constitute the top or the bottom of a diagram involving $n$ strands.
\smallbreak

Now quotient this category by the following relations (compare with the reduced presentation of the virtual
braid group in Proposition 1).

\begin{enumerate}
\item $\mu \mu' = id_{* \otimes *} = \mu' \mu ,$

\item $v v = id_{*},$

\item $\mu_{12} v_j=v_j \mu_{12}, \ \ \  \mbox{for} \  j>2 , $

\item $ \mu_{12} \, v_2 \mu_{12} v_2 \, v_1v_2 \mu_{12} v_2 v_1 =
 v_1v_2 \mu_{12} v_2 v_1 \, v_2 \mu_{12} v_2  \, \mu_{12} , $

 \item $\mu_{12}\, v_2 v_3 v_1 v_2  \mu_{12} v_2 v_1 v_3 v_2 =
v_2 v_3 v_1 v_2  \mu_{12} v_2 v_1 v_3 v_2  \, \mu_{12},   $

\item $v_{i} v_{i+1} v_{i}  =  v_{i+1} v_{i}  v_{i+1}, $

\item $v_i v_j =  v_j v_i,  \mbox{for} \ j\neq i\pm 1.$

\end{enumerate}
This quotient is called the {\em String Category} and denoted $SC.$ The category $SC$ is still strict
monoidal.
\end{defn}

To recapture the connecting string morphisms $\mu_{ij}$ in the String Category context, we follow the formalism of the previous sections. Define
$$\mu_{i,i+1} = id_{*} \otimes \cdots   \otimes id_{*} \otimes \mu \otimes  id_{*} \otimes  \cdots  \otimes id_{*}$$ where
$\mu$ occurs in the $i$ and $i+1$ places in the tensor product and define
$$\mu_{i+1,i} = id_{*} \otimes  \cdots  \otimes id_{*} \otimes \mu' \otimes  id_{*} \otimes  \cdots   \otimes id_{*}$$ where
$\mu'$ occurs in the $i$ and $i+1$ places in the tensor product.
Define, for $i<j$, the element $\mu_{ij}$ by the formula
\begin{equation}\label{catmuij}
\mu_{ij} = v_{j-1} v_{j-2} \ldots v_{i+1} \, \mu_{i,i+1} \, v_{i+1} \ldots v_{j-2} v_{j-1}.
\end{equation}
and define
\begin{equation}\label{catmuji}
\mu_{ji} = v_{j-1} v_{j-2} \ldots v_{i+1} \, \mu_{i+1,i} \, v_{i+1} \ldots v_{j-2} v_{j-1}.
\end{equation}

\begin{rem} \rm
In this notation, relation $(4)$ in Definition \ref{sc} becomes the algebraic Yang-Baxter equation
$$\mu_{12} \mu_{13} \mu_{23} = \mu_{23} \mu_{13} \mu_{12},$$ and relation $(5)$ becomes the
commuting relation $$\mu_{12} \mu_{34} = \mu_{34} \mu_{12}.$$

Then one has, as consequences,  the general algebraic Yang-Baxter equation and commuting relations, as we have described them in earlier sections of the paper:
$$\mu_{ij}\mu_{ik}\mu_{jk} = \mu_{jk}\mu_{ik}\mu_{ij}, \quad \text{for all distinct } i,j,k$$
and
$$\mu_{ij}\mu_{kl} = \mu_{kl}\mu_{ij}, \quad \{i,j\} \cap \{k,l\} = \emptyset.$$

  Diagrammatically, $\mu_{ij}$ consists in
$n$ parallel strands with a string connector between the $i$-th and $j$-th strands directed from
 $i$ to $j.$ Similarly, $v_{i}$ corresponds to a diagram of $n$ strands where there is a virtual crossing
 between the $i$-th  and $(i+1)$-st strands. An $n$-strand diagram that is a product of
these generators is regarded as a morphism from $[n]$ to $[n]$ for $n$ any natural number.  We interpret $\mu_{ij}$ and $v_{i}$ diagrammatically
according to the conventions previously established in this paper.
\end{rem}

The morphisms $v_{i}$ effect the action of the symmetric group and the category models the
 virtual braid group in the following precise sense.

\begin{thm}
The virtual braid group on
$n$ strands is isomorphic to the group of morphisms from $[n]$ to $[n]$  in the String Category.
\end{thm}

\begin{proof}
By Proposition 1, for any positive integer $n,$ the group of endomorphisms of the object
$ [n] = *^{\otimes n}$ is isomorphic to $VB_{n}.$
\end{proof}

The point of this categorical formulation of the virtual braid groups is that we see how these groups form a natural extension of the symmetric groups by formal elements that satisfy the algebraic Yang-Baxter equation.
The category we desribe is a natural structure for an algebraist interested in exploring formal properties of the algebraic Yang-Baxter equation. It should be remarked that the relationship between the relations
in the virtual pure braid group and the algebraic Yang-Baxter equation was also pointed out in
\cite{BER}. See also \cite{Bellingeri} Remark 10. We have taken this observation further to point out that the virtual braid group is a direct result of forming a convenient category associated with the algebraic Yang-Baxter equation.
\smallbreak

For the reader who would like to take the String Category $SC$ as a starting point for the theory of virtual
braids, here is a description of how to read our figures for that purpose. Figure~\ref{genrs} illustrates the
permutation generators $v_{i}$ for the String Category.  The braiding elements  $\sigma_{i}$ will be
defined in terms of the string generators. Elementary connecting strings are given in Figure~\ref{elstrings}.
It is implicit in Figure~\ref{elstrings} how to define the braiding elements $\sigma_{i}$  by composing
string generators with permutations (virtual crossings). See also Figure~\ref{relstrings}, which  illustrates basic relationships among
string generators, permutations and braiding operators. Figure~\ref{longstrings} illustrates the general connecting strings and their relations with the permutation operators. In particular, Figure~\ref{longstrings} shows how any string
connection can be written in terms of a basic string generator and a product of permutations.
Figure~\ref{exchange}
illustrates how $\mu_{ij}$ and $\mu_{ji}$ are related diagrammatically. Figures~\ref{slides},~\ref{slidepf} and~\ref{commute}  show the basic slide relations between string connections and permutations.
Figure~\ref{sbr} illustrates the
algebraic Yang-Baxter relation as it occurs for the string connectors.

\section{Representations of the Virtual and Pure Virtual Braid Groups}

\subsection{}

Let $A$ be an algebra over a ground ring $k$. Let $\rho \in A \otimes A$ be an element of the tensor product of $A$ with itself. Then $\rho$ has the form given by the following equation
\begin{equation}
\rho = \sum_{i = 1}^{N} e_{i} \otimes e^{i}
\end{equation}
where $e_{i}$ and $e^{j}$ are elements of the algebra $A$.
We will write this sum symbolically as
\begin{equation}
\rho = \sum e \otimes e'
\end{equation}
where it is understood that this is short-hand for the above specific summation.

We then define, for $i < j$,  $\rho_{ij} \in A^{\otimes n}$ by the equation
\begin{equation}
\rho_{ij} = \sum 1_{A} \otimes \cdots \otimes 1_{A} \otimes  e \otimes 1_{A} \otimes \cdots \otimes 1_{A} \otimes e' \otimes 1_{A} \otimes \cdots \otimes 1_{A}
\end{equation}
where the $e$ occurs in the $i$-th tensor factor and the $e'$ occurs in the $j$-th tensor factor.

 With $i < j$ we also define $\rho_{ji}$ by reversing the roles of $e$ and $e'$ as shown in the next equation
\begin{equation}
\rho_{ij} = \sum 1_{A} \otimes \cdots \otimes 1_{A} \otimes  e' \otimes 1_{A} \otimes \cdots \otimes 1_{A} \otimes e \otimes 1_{A} \otimes \cdots \otimes 1_{A}
\end{equation}
where $e'$ occurs in the $i$-th tensor factor and $e$ occurs in the $j$-th tensor factor.

We say that $\rho$ is a {\it solution to the algebraic Yang-Baxter equation} if it satisfies, in
 $A^{\otimes n}$ for $n \geq 3,$  the equation
\begin{equation} \label{ybe}
\rho_{12} \rho_{13} \rho_{23} = \rho_{23} \rho_{13} \rho_{12}.
\end{equation}

It is immediately obvious that if $\rho$ satisfies the algebraic Yang-Baxter equation, then, for any
pairwise distinct $i,  j, k$ we have
\begin{equation}
\rho_{ij} \rho_{ik} \rho_{jk} = \rho_{jk} \rho_{ik} \rho_{ij}.
\end{equation}
This gives all possible versions of the  algebraic Yang-Baxter equation occuring in the tensor product
$A^{\otimes n}.$
\smallbreak

The following proposition is an immediate consequence of our presentation for the pure virtual braid group.

\begin{prop} Let $VP_{n}$ denote the pure virtual braid group with generators $\mu_{ij}$ and relations as given in
Theorem~2 of Section~3. Let $A$ be an algebra with an invertible algebraic solution to the Yang-Baxter equation denoted by
$\rho \in A \otimes A$ as described above. Define $$rep: VP_{n} \longrightarrow A^{\otimes n}$$ by the equation
\[
rep(\mu_{ij}) = \rho_{ij}.
\]
Then $rep$ extends to a representation of the the virtual braid group to the tensor algebra $A^{\otimes n}$.
\end{prop}

\begin{proof}
It follows at once from the definitions of the $\rho_{ij}$ that $\rho_{ij}\rho_{kl} = \rho_{kl}\rho_{ij}$ whenever the sets
$\{i,j\}$ and $\{k,l\}$ are disjoint. Thus, we have shown that the $\rho_{ij}$ satisfy all the relations in the pure   virtual  braid group.
This completes the proof of the Proposition.
\end{proof}

Next, we show how to obtain representations of the full virtual braid group. To this purpose, consider the algebra $Aut(A^{\otimes n})$ of
linear automorphisms of $A^{\otimes n}$ as a module over $A$. Assume that we are given an invertible
solution to the algebraic Yang-Baxter equation,
$\rho \in A \otimes A$, and define $\tilde{\rho}_{ij}:A^{\otimes n} \longrightarrow A^{\otimes n} $ by the equation
$\tilde{\rho}_{ij}( \alpha ) = \rho_{ij} \alpha$ where $\alpha \in A^{\otimes n}.$ Since $\rho$ is invertible, $\tilde{\rho}_{ij} \in Aut(A^{\otimes n})$. Let $P_{ij} : A^{\otimes n} \longrightarrow A^{\otimes n}$ be the mapping that interchanges the $i$-th and $j$-th tensor
factors.  Note that  $P_{ij} \in Aut(A^{\otimes n}).$ We let $P_{i}$ denote $P_{i,i+1}.$  We now define
$$Rep: VB_{n} \longrightarrow Aut(A^{\otimes n})$$ by
the equations
$$Rep(\mu_{ij}) = \tilde{\rho}_{ij} \ \ \ \mbox{and} \ \ \  Rep(v_{i}) = P_{i}.$$
The next proposition is a consequence of  presentation~(\ref{stringypres}) for the virtual braid group.

\begin{prop} The mapping $Rep:VB_{n} \longrightarrow Aut(A^{\otimes n})$, defined above, is a representation of the virtual braid group to a subgroup of $Aut(A^{\otimes n}).$
\end{prop}

\begin{proof} It is clear that the elements $P_{i}$ obey all the relations in the symmetric group $S_{n}$.
By presentation~(\ref{stringypres}) it remains to show that letting $\lambda = Rep(\tau)$ where $\tau$ is an element of
$S_{n}$, the relations
$$\lambda \rho_{ij} \lambda^{-1} =  \tilde{\rho}_{\tau(i), \tau(j)}, \ \ \  \tau \in S_{n}$$
are satisfied in $Aut(A^{\otimes n}).$
Since $\rho_{ij}$ is defined via the placement of the $e$ and $e'$ factors in the summation for $\rho$ on the $i$-th and $j$-th strands, these
relations are immediate. This completes the proof of the proposition.
\end{proof}

\begin{rem} \rm The method we have described for constructing a representation of the
virtual braid group from an algebraic solution to the Yang-Baxter equation generalizes the well-known
construction of a representation of the classical Artin braid group from a solution to the Yang-Baxter equation in braided form. In the usual method for constructing the classical representation, one composes the algebraic solution with a permutation, obtaining a solution to the braiding equation
$(B1)$. This composition is the same as our relation
$$
\sigma_{i}  = \mu_{i,i+1} v_{i}
$$
between the braiding element $\sigma_{i}$ and the stringy generator
$\mu_{i,i+1}$ for the pure virtual  braid group. Without the concept of virtuality, the direct relationship
of the algebraic Yang-Baxter equation with the braid groups would not be apparent. We see that, from an algebraic point of view, the virtual braid group is an entirely natural construction. It is the universal
algebraic structure related to viewing solutions to the algebraic Yang-Baxter equation inside tensor products of algebras and endowing these tensor products with the natural permutation action of the symmetric group.
\end{rem}

Solutions to the algebraic Yang-Baxter equation are usually thought of as deformations of the identity mapping on a two-fold tensor product $A \otimes A.$ We think of a braiding operator as a deformation of
a transposition, and so one goes between the algebraic and braided versions of such operators by
composition with a transposition.
\smallbreak

The Artin braid group $B_{n}$ is motivated by a combination of topological considerations and the desire for a group structure that is very close to the structure of the symmetric group $S_{n}$.
We have seen that the virtual braid group $VB_{n}$ is motivated at first by a natural extension of
the Artin braid group in the context of virtual knot theory, but now we see a different motivation for
the virtual braid group. Given that one studies the algebraic Yang-Baxter equation in the
context of tensor powers of an algebra $A$, it is thoroughly natural to study the compositions
of algebraic braiding operators placed in two out of the $n$ tensor lines (the stringy generators) and to
let the permutation group of the tensor lines act on this algebra. As we have seen in (\ref{stringypres}), this
is precisely the virtual braid group. Viewed in this way, the virtual braid group has nothing to do with the plane and nothing to do with virtual crossings. It is a natural group associated with the structure of
algebraic braiding.

\subsection{A Representation Category for the String Category}

We now give a categorical interpretation of virtual knot theory and the virtual braid group in terms
of representation modules associated to an algebra $A$ over a commutative ring $k$ with an algebraic Yang-Baxter element  $\rho$ as above.
Let  $End(A^{\otimes n})$ denote the linear endomorphisms of  $A^{\otimes n}$ as a module over $A$.
View $End(A^{\otimes n})$ as the set of morphisms in  a category $Mod^{n}_{k}$  with  $A^{\otimes n}$ as the single object.  We single out the following morphisms in this category:
\begin{enumerate}
\item $\alpha_{1} \otimes \alpha_{2} \otimes \cdots \otimes \alpha_{n} \in A^{\otimes n}$ acting
on $A^{\otimes n}$ by left multiplication,
\item the elements of the symmetric group $S_{n}$, generated by transpositions of adjacent
tensor factors.
\end{enumerate}
In making the representation of $VB_{n}$ we have used the
stringy generators $\mu_{ij}$ and mapped them to sums of morphisms of the first type above.
The virtual braid group $VB_{n}$ described via (\ref{stringypres}),  can be viewed as a category with one
object and generators $\mu_{ij}$ and $v_{k}.$ We let  $Mod_{k}$ denote the category that is obtained
by taking all of the categories $Mod^{n}_{k}$ together with objects $A^{\otimes n}$ for each natural number $n$ and morphisms from all of the $End(A^{\otimes n}).$

\begin{rem} \rm
Of course any associative algebra can be seen as
a single object category with morphisms the elements of the algebra. But here we have a pictorial
representation of the morphisms as stringy braid diagrams. These diagrams, which capture the pure virtual braid group so far,  can be generalized
by taking the transpositions of the form $P_{i, i+1}$  via a diagram of lines
$i$ and $i+1$ crossing through one another to form virtual crossings $v_{i}.$
Seen from the categorical view that we have developed in these last sections, the virtual crossings are interpreted as generators of the symmetric group whose action is added naturally to the algebraic structure of the pure virtual braid group. By bringing in this action, we expand the pure virtual braid group to the virtual braid group. The virtual crossings  have thus become part of the embedded symmetry of the structure of the virtual braid group. This is in sharp contrast to the role of the virtual crossings in the original form of the virtual knot theory. There the virtual crossings appear
as artifacts of the presentation of virtual knots in the plane where those knots acquire extra crossings
that are not really part of the essential structure of the virtual knot. Nevertheless, these same crossings
appear crucially in the virtual braid group, and turn into the generators of the symmetric group embedded in the virtual braid group. With the use of the full set of $\mu_{ij}$ in (\ref{stringypres}) the detour moves and other remnants of the virtual crossings as artifacts have completely disappeared into the permutation action. We will continue the categorical discussion for the virtual braid group, after first discussing certain aspects of knot theory and the tangle categories.
\end{rem}

We can now state a general representation theorem.
\smallbreak

\begin{thm} Any monoidal functor
$$F: SC \longrightarrow Mod_{k}$$
gives rise to a representation of $VB_{n}:$
$$f \in End_{SC}([n]) \simeq VB_{n} \longmapsto F(f) \in End_{k}(A^{\otimes n})$$
where $A = F(*).$
\end{thm}

\begin{proof} The proof follows from the previous discussion.
\end{proof}

The representations of $VB_{n}$ that we have here derived can be interpreted as follows.

\begin{thm} Let $\rho \in A \otimes A$ be a solution of the algebraic Yang-Baxter equation, where
$A$ is an algebra over a commutative ring $k.$ One can then define a monoidal functor
$$F_{A}:SC \longrightarrow Mod_{k}$$ by setting $F_{A}(*) = A$, $F_{A}(\mu) = \tilde{\rho}$,
and $F_{A}(v_{i}) = P$, where the endomorphisms $\tilde{\rho}$ and $P$ of $A \otimes A$ are
given by $$\tilde{\rho}(x \otimes y) = \rho(x \otimes y)$$ and $$P(x \otimes y) = y \otimes x$$
for all $x, y \in A.$
\end{thm}

\begin{proof} The proof follows from the previous discussion.
\end{proof}

\subsection {Virtual Hecke Algebra}

>From the point of view of the theory of braids the Hecke algebra $H_{n}(q)$ is a quotient of the
group ring ${\Bbb Z}[q,q^{-1}][B_{n}]$ of the Artin braid group by the ideal generated by the quadratic expressions
\begin{equation} \label{HeckeQuad}
\sigma_{i}^2 - z\sigma_{i} - 1
\end{equation}
for $i = 1, 2, \ldots n-1,$ where $z = q - q^{-1}.$
This corresponds to the identity $\sigma_{i} - \sigma_{i}^{-1} = z1$,
which is sometimes regarded diagrammatically as a  skein identity for calculating knot polynomials.
By the same token, we define the {\it virtual Hecke algebra} $VH_{n}(q)$ to be the quotient of the group ring
${\Bbb Z}[q,q^{-1}][VB_{n}]$ by the ideal generated by Eqs.~\ref{HeckeQuad}.
\smallbreak

There are difficulties in extending structure theorems for the Hecke algebra to corresponding structure
theorems for the virtual Hecke algebra, such as finding normal forms, studying the representation theory
and constructing Markov traces.
Yet, some matters of representations do generalize directly.
In particular, let $W$ be a module over ${\Bbb Z}[q, q^{-1}]$
and let $I:W \longrightarrow W$ be the identity operator.
If $R:W \otimes W \longrightarrow W \otimes W$ is a solution to the Yang-Baxter equation satisfying
$$R^{2} = zR + I,$$ then one has a corresponding  representation
$T: VH_{n}(q) \longrightarrow Aut(W^{\otimes n})$. This representation is specified as follows.
\begin{equation}
T(\sigma_{i} ) = \sum 1 \otimes \cdots \otimes 1 \otimes  R \otimes 1 \otimes \cdots \otimes 1
\end{equation}
where $R$ operates  on the $i$-th and $(i+1)$-st tensor factors, and
\begin{equation}
T(v_{i})  = \sum 1 \otimes \cdots \otimes 1 \otimes  P \otimes 1 \otimes \cdots \otimes 1
\end{equation}
where $P$ acts by permuting  the $i$-th and $(i+1)$-st tensor factors. It is easy to see that this gives a representation of the virtual Hecke algebra.
\smallbreak

One can hope that the presence  of such representations would shed light on the existence of
a generalization of the Ocneanu trace \cite{Jones} on the Hecke algebra to a corresponding trace
and link invariant using the virtual Hecke algebra. At this point there is an issue about the nature of the
generalization. One can aim for a trace on the virtual Hecke algebra that is compatible with the
Markov Theorem for virtual knots and links as formulated in \cite{Kamada,KL4}. This means that the trace must be compatible with both classical and virtual stabilization. This is a trace that
is difficult to achieve. A simpler trace is possible by working in rotational virtual knot theory where virtual stabilization is not allowed \cite{VKT}. See the next section for a discussion of unoriented quantum invariants for rotational virtuals.
We will report on the relation of this approach with the Markov Theorem for virtual knots and links in a separate paper.
\smallbreak

Another line of investigation is suggested by translating the basic Hecke algebra relation into the language of stringy connections. We have $\sigma = \mu v$ for the abstract relation between a braiding generator, a connector and a virtual element. Thus, the virtual Hecke relation
$\sigma^{2} = z \sigma + 1$
becomes $$\mu v \mu = z \mu + v,$$ and it is possible to work in the presentation (20) of the virtual
braid group to find a structure theory for the virtual Hecke algebra.

\section {Rotational Virtual Links, Quantum Algebras, Hopf algebras and the Tangle Category}

This section will show how the ideas and methods of this paper fit together with representations
of quantum algebras (to be defined below) and Hopf algebras and invariants of virtual links.
We begin with a quick review of the theory of virtual links (in relation to virtual braids), and we construct
the virtual tangle category. This category is a natural generalization of the virtual braid group. A functor from the virtual tangle category to an algebraic category will form a generalization of the representations of virtual braid groups that we have discussed in the previous section. This functor
is related to (rotational) invariants of virtual knots and links. It is not hard to see that the construction given in this section defines a category (for arbitrary Hopf algebras) that generalizes the String Category given earlier in this paper. The category that we define here contains virtual crossings, special elements that satisfy the algebraic Yang-Baxter equation and also cup and cap operators. The subcategory
without the cup and cap operators and without any (symbolic) algebra elements except those involved with the algebraic Yang-Baxter operators is isomorphic to the String Category.
\smallbreak

A word to the reader about this section: In one sense this section is a review of known material in the form that Kauffman and Radford \cite{KRH} have shaped the theory of quantum invariants of knots and three-manifolds via finite-dimensional Hopf algebras. On the other hand, this theory is  generalized here to invariants of rotational virtual knots and links. This generalization is new, and it is directly related to the structure of the virtual braid group as described in the earlier part of this paper. We have given a complete sketch of this generalization.
The reader should take the word {\it sketch} seriously and concentrate on the sequence of diagrams that depict the ingredients of the theory. Taking this point of view, the reader can see that the appearance of the algebraic Yang-Baxter element in our diagrams (See Figure~\ref{ApplyFunctor}) is aided by using a connecting string exactly analogous to the connecting string in the earlier part of the paper. The generalization follows by taking the functorial image of the virtual tangle category defined in this section.
\smallbreak

\subsection{Virtual Diagrams}
We begin with Figure~\ref{vmoves}. This figure illustrates the moves on virtual knot and link
diagrams that serve to define the theory of virtual knots and links. Two knot or link diagrams
with virtual and classical crossings are said to be {\it virtually isotopic} if one can be obtained from
the other by a finite sequence of these moves. In the figure the moves are divided into
type A, B and C moves. Moves of type A are the classical Reidemeister moves. These
are essentially the same as corresponding moves in the Artin braid group except for the boxed
move involving a loop in the diagram. The move involving this loop is usually called the
{\it first Reidemeister move.} When we forbid the first Reidemeister move, the equivalence relation is
called {\it regular isotopy}. The moves of type B are purely virtual and (except for the move involving
a virtual loop)  correspond to the properties of virtual crossings in the virtual braid group. We call the
equivalence relation that forbids both the virtual loop move and the classical loop move
{\it virtual regular isotopy}. Finally, we have moves of type C. These are the local detour moves, and
they correspond to the mixed moves in the virtual braid group.
\smallbreak

\begin{figure}
     \begin{center}
     \begin{tabular}{c}
     \includegraphics[width=10cm]{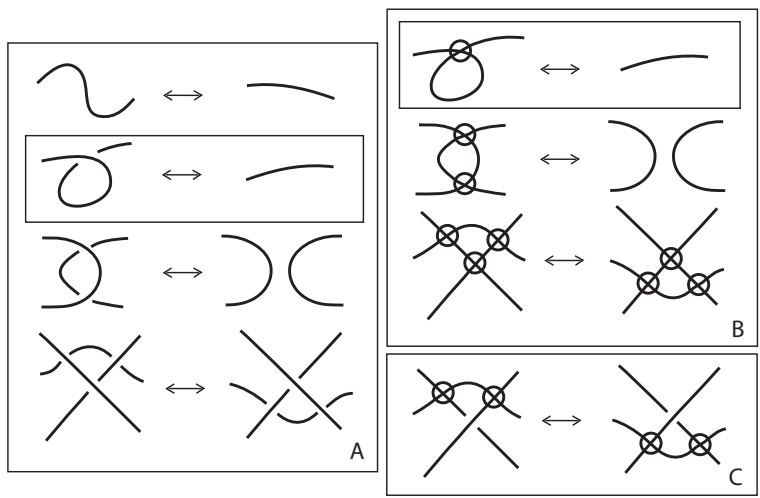}
     \end{tabular}
     \caption{Virtual moves}
     \label{vmoves}
\end{center}
\end{figure}

In this section we will work with virtual knots and links up to virtual regular  isotopy. In addition to the
usual kinds of virtual phenomena, we will see some extra features in looking at this equivalence relation.
 Two virtual knot or link diagrams are said to be {\it rotationally equivalent} if they are equivalent under virtual regular isotopy.  {\it Rotational virtual knot theory}
is the study of the rotational equivalence classes of virtual knot and link diagrams. Studied under this equivalence relation, virtual knot and link diagrams are called {\it rotational virtuals.} We shall say that a virtual knot or link is {\it rotationally knotted} or {\it rotationally linked} if it is not equivalent to an unknot or an unlink under virtual regular isotopy.
View Figure~\ref{rotateknot} and Figure~\ref{rotatelink}. In the first figure we illustrate a rotational
virtual knot, and in the second we show a rotational virtual link. Both the knot and the link are kept from
being trivial by the presence of flat loops as discussed above. There is much more to say about rotational virtuals, and we refer the reader to \cite{VKT} for some steps in this direction.
\smallbreak

\begin{figure}
     \begin{center}
     \begin{tabular}{c}
     \includegraphics[width=2.5cm]{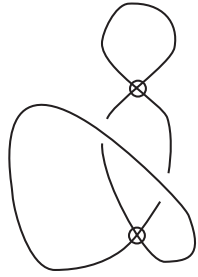}
     \end{tabular}
     \caption{A rotational virtual knot}
     \label{rotateknot}
\end{center}
\end{figure}

\begin{figure}
     \begin{center}
     \begin{tabular}{c}
     \includegraphics[width=3.3cm]{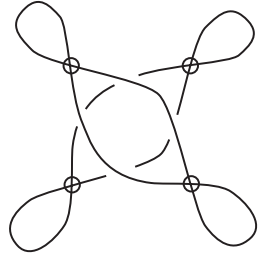}
     \end{tabular}
     \caption{A rotational virtual link}
     \label{rotatelink}
\end{center}
\end{figure}

\subsection{The Virtual Tangle Category}
The advantage in studying virtual knots up to virtual regular isotopy is that all so-called quantum link
invariants generalize to invariants of virtual regular isotopy. This means that virtual regular isotopy is
a natural equivalence relation for studying topology associated with solutions to the Yang-Baxter
equation.
\smallbreak

Here we create a context  by defining the {\it Virtual Tangle Category}, $VTC,$ as indicated
in Figure~\ref{regtang}.
The tangle category is generated by the morphisms
shown in the box at the top of this figure. These generators are: a single identity line, right-handed and left-handed
crossings, a cap and a cup, a virtual crossing. The objects in the tangle category consist in the set of
$[n]$'s  where $n = 0, 1,2, \ldots.$ For a morphism $[n] \longrightarrow [m]$, the numbers $n$ and $m$ denote, respectively, the number of free arcs at the bottom and at the top of the diagram that represents  the morphism. The morphisms are like braids except that they can (due to the presence of the cups and caps) have different numbers of free ends at the top and the bottom of their diagrams.
\smallbreak

The sense in which the elementary morphisms (line, cup, cap, crossings)  generate the tangle category is composition as  shown in Figure~\ref{tangprod}. For composition, the segments are matched so that the number of lower free ends on each segment is equal to the number of upper free ends on the segment below it.
The  Figure~\ref{tangprod} shows a virtual trefoil as a morphism from
$[0]$ to $[0]$ in the category.
The tensor product of morphisms is the horizontal juxtaposition of their
diagrams. Each of the seven horizontal segments of the figure represents one of the elementary morphisms tensored  with the identity line. Consequently there is a well-defined composition of all of the segments and this composition is a morphism $[0] \longrightarrow [0]$ that represents the knot.
\smallbreak

The basic equivalences of morphisms are
shown in Figure~\ref{regtang}. Note that $II, III, V$ are formally equivalent to the rules for unoriented
virtual braids. The zero-th move is a cancellation of consecutive maxima and minima, and the move $IV$
is a swing move in both virtual and classical relations of crossings to maxima and minima. It should be
clear that the tangle category is a generalization of the virtual braid group with a natural inclusion of unoriented virtual braids as special tangles in the category. Standard braid closure and the plat closure of braids
have natural definitions as tangle operations. Any virtual knot or link can be
represented in the tangle category as a morphism from $[0]$ to $[0],$ and one can prove that {\it two virtual links are
virtually regularly isotopic if and only if their tangle representatives are equivalent in the tangle category.}
None of the rules for equivalence in the tangle category involve either a classical loop or a
virtual loop. This means that the virtual tangle category is a natural home for the theory of rotational
virtual knots and links.
\smallbreak

\begin{figure}
     \begin{center}
     \begin{tabular}{c}
     \includegraphics[width=8cm]{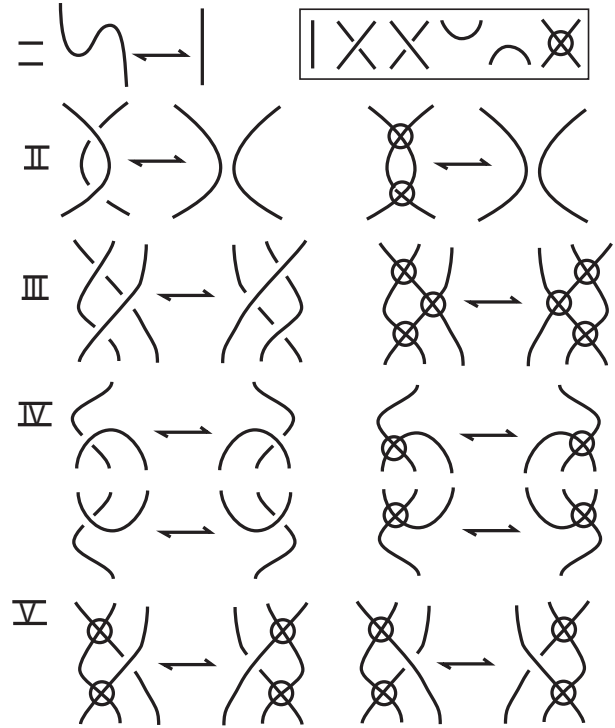}
     \end{tabular}
     \caption{Regular isotopy with respect to the vertical direction}
     \label{regtang}
\end{center}
\end{figure}

\begin{figure}
     \begin{center}
     \begin{tabular}{c}
     \includegraphics[width=6cm]{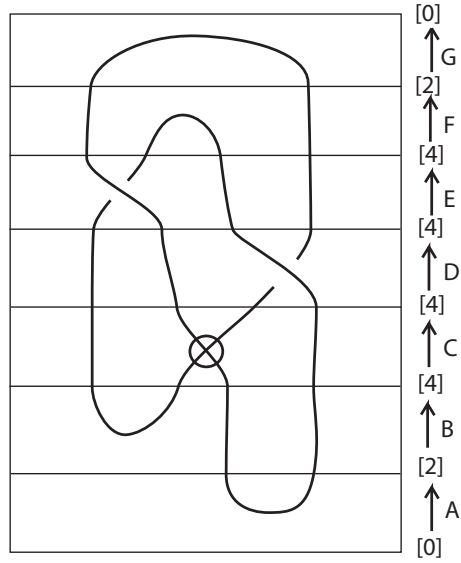}
     \end{tabular}
     \caption{Virtual trefoil as a morphism in the tangle category}
     \label{tangprod}
\end{center}
\end{figure}

\subsection{Quantum Algebra and Category}
Now we shift to a category associated with an  algebra that is directly related to our representations
of the virtual braid group. We take the following definition \cite{KP, KRH}: A {\it quantum algebra} $A$
is an algebra over a commutative ground ring $k$ with an invertible  mapping $s: A \longrightarrow A$
that is an {\it antipode}, that is $s(ab) = s(b)s(a)$ for all $a$ and $b$ in $A,$ and there is an element
$\rho \in A \otimes A$ satisfying the algebraic Yang-Baxter equation as in Equation~\ref{ybe}:
$$\rho_{12} \rho_{13} \rho_{23} = \rho_{23} \rho_{13} \rho_{12}.$$
We further assume that $\rho$ is invertible and that
$$ \rho^{-1} = (1_{A} \otimes s)\circ \rho = (s \otimes 1_{A})\circ \rho.$$
The multiplication in the algebra is usually denoted by $m: A\otimes A \longrightarrow A$ and is
assumed to be associative. It is also assumed that the algebra has a multiplicative unit element.
The defining properties of a quantum algebra are part of the properties of a Hopf algebra, but a Hopf algebra has a comultiplication $\Delta: A \longrightarrow A \otimes A$
that is a homomorphism of algebras, plus a list of further relations, including a fundamental relationship between the multiplication, the comultiplication and the antipode. In the interests of simplicity, we shall restrict ourselves to quantum algebras here, but most of the remarks that follow apply to Hopf algebras, and particularly quasi-triangular Hopf algebras. Information on Hopf algebras is included at the end of this section.  See \cite{KRH} for more about these connections.
\smallbreak

We construct a category $Cat(A)$ associated  with a quantum algebra $A$. This category is a very close relative to the virtual tangle category. $Cat(A)$ differs from the tangle category in that it has only virtual crossings, and
there are labeled vertical lines that carry elements of the algebra $A.$  See Figure~\ref{CatMorph}.
Each such labeled line is a morphism in the category.  The virtual crossing is a generating morphism as are the cups, caps and labeled lines. The objects in this category are the same
entities $[n]$ as in the tangle category. This category is identical in its framework to the tangle category but the crossings are not
present and lines labeled with algebra are present.
Given $a,b \in A$ we compose the morphisms corresponding to $a$ and $b$ by taking a line labeled
$ab$ to be their composition. In other words, if $\langle x \rangle$ denotes the morphism in
$Cat(A)$ associated with $x \in A$, then
$$\langle a \rangle \circ \langle b \rangle = \langle ab \rangle.$$
As for the additive structure in the algebra, we extend the category to an additive category by formally adding the generating morphisms (virtual crossings, cups, caps and algebra line segments). In Figure~\ref{CatMorph} we illustrate the composition of
such morphisms and we illustrate a number of other defining features of the category $Cat(A).$
\smallbreak

\begin{figure}
     \begin{center}
     \begin{tabular}{c}
     \includegraphics[width=8cm]{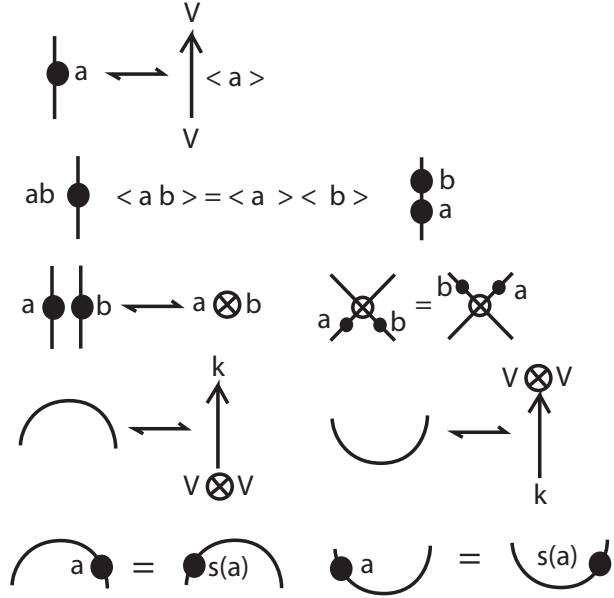}
     \end{tabular}
     \caption{ Morphisms in $Cat(A)$}
     \label{CatMorph}
\end{center}
\end{figure}

In the same figure we illustrate how the tensor product of elements $a \otimes b$ is represented by parallel vertical
lines with $a$ labeling the left line and $b$ labeling the right line. We indicate that the virtual
crossing acts as a permutation in relation to the tensor product of algebra morphisms. That is, we
illustrate that $$\langle a \rangle \otimes \langle b \rangle \circ P = P \circ \langle b \rangle \otimes \langle a \rangle.$$   Here $P$ denotes the virtual crossing of two segments, and is regarded as a morphism $P: V \otimes V \longrightarrow  V \otimes V$ (see remark below). Since the lines interchange, we expect $P$ to behave as the permutation of the two tensor factors.
\smallbreak

In Figure~\ref{CatMorph} we show the notation
$V$ for the object $[1]$ in this category and we use $V \otimes V = [2]$, $V \otimes V \otimes V= [3]$
and so on for all the natural number objects in the category. We write
$[0] = k$, identifying the ground ring with the ``empty object" $[0].$ It is then axiomatic that
$k \otimes V = V \otimes k = V.$ Morphisms are indicated both diagrammatically and in terms of arrows
and objects in this figure. Finally, the figure indicates the arrow and object forms of the cup and the
cap, and crucial axioms relating the antipode with the cup and the cap.
A cap is regarded as a morphism from $V \otimes V$ to $k$, while
a cup is regarded as a morphism form $k$ to $V \otimes V.$
The basic property of the cup and the cap is the  {\em Antipode Property: if
one ``slides" a decoration across the maximum or minimum in a counterclockwise
turn, then the antipode $s$ of the algebra is applied to the decoration.}
In categorical terms this property says $$Cap \circ (\langle 1 \rangle \otimes a) = Cap \circ
(\langle sa \rangle \otimes 1 )$$ and $$(\langle a \rangle \otimes 1) \circ Cup  =  (1 \otimes \langle sa \rangle ) \circ Cup.$$  Here $1$ denotes the identity morphism for $[0]$. These properties and other naturality properties of the cups and the
caps are illustrated in Figure~\ref{CatMorph} and Figure~\ref{antipode}.  The naturality properties of the flat diagrams  in this category include regular homotopy of immersions (for diagrams without algebra decorations), as illustrated in these figures.
\smallbreak

In Figure~\ref{antipode}  we see how the antipode property of the cups and caps leads to a
diagrammatic interpretation of the antipode. In the figure we see that the antipode $s(a)$ is represented by composing with a cap and a cup on either side of the morphism for $a$. In terms of the composition
of morphisms this diagram becomes
$$\langle sa \rangle = (Cap \otimes 1)  \circ (1 \otimes \langle a \rangle \otimes 1)\circ(1 \otimes  Cup).$$
Similarly, we have
$$\langle s^{-1}a \rangle = (1 \otimes Cap)  \circ(1 \otimes \langle a \rangle \otimes 1)\circ( Cup \otimes 1).$$
This, in turn, leads to the
interpretation of the flat curl as an  element $G$ in $A$ such that
$s^{2}(a) = GaG^{-1}$ for all $a$ in $A.$  $G$ is a flat curl diagram interpreted
as a morphism in the category. We see that, formally, it is natural to interpret
$G$ as an element of $A$.   In a so-called {\em ribbon Hopf
algebra} there is such an element  already in the algebra. In the general case it
is natural to extend the algebra to contain such an element.
\smallbreak

\begin{figure}
     \begin{center}
     \begin{tabular}{c}
     \includegraphics[width=8cm]{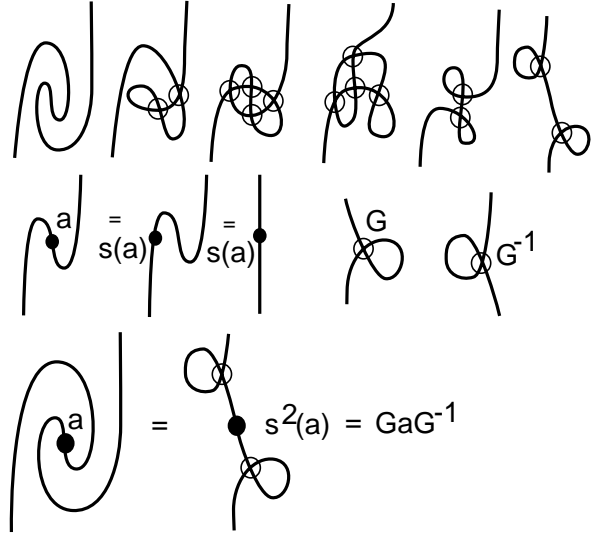}
     \end{tabular}
     \caption{Diagrammatics of the antipode}
     \label{antipode}
\end{center}
\end{figure}

\begin{figure}
     \begin{center}
     \begin{tabular}{c}
     \includegraphics[width=8cm]{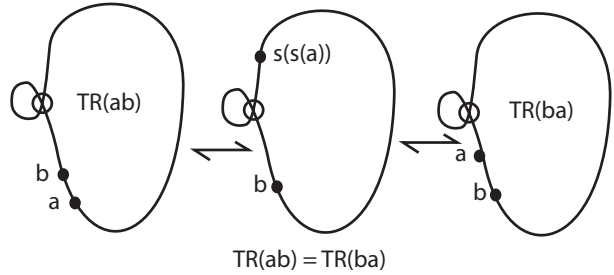}
     \end{tabular}
     \caption{Formal trace}
     \label{trace}
\end{center}
\end{figure}

\subsection{The Basic Functor and the Rotational Trace}
We are now in a position to describe a functor $F$ from the  virtual tangle category $VTC$
to $Cat(A).$  (Recall that the virtual tangle category is defined for virtual link diagrams without
decorations. It has the same objects as $Cat(A).$)
$$F:VTC \longrightarrow Cat(A)$$
The functor $F$ decorates each
positive  crossing of the tangle (with respect to the vertical - see Figure~\ref{Functor})
with the Yang-Baxter element (given by the quantum algebra $A$)
$\rho = \Sigma e \otimes e^{'}$ and each negative crossing (with respect to the
vertical) with $\rho^{-1} = \Sigma s(e) \otimes e^{'}$. The form of the
decoration is indicated in Figure~\ref{Functor}. Since we have labelled the negative crossing with the
inverse Yang-Baxter element, it follows at once that the two crossings are mapped to inverse elements in the category of the algebra.  {\it This association is a direct generalization of our mapping
of the virtual braid group to the stringy connector presentation.}
\smallbreak

We now point out the structure of the image of a knot, link or tangle under this functor.
The key point about this functor is that, because quantum algebra elements can be
moved around the diagram, we can concentrate all the image algebra in one place.
Because the flat curls are identified with either $G$ or $G^{-1}$, we can use
regular homotopy of  immersions to bring the image under $F$ of each component of a virtual link diagram to the form of a circle with a single concentrated decoration (involving a sum over many
products) and a reduced pattern of flat curls that can be encoded as a power of the special
element $G.$ Once the underlying curve of a link  component is converted to a loop with
total turn zero, as in Figure~\ref{trace}, then we can think of such a loop, with algebra labeling the
loop, as a representative for a formal trace of that algebra and call it $TR(X)$ as in the figure.
In the figure we illustrate that for such a labeling $$TR(ab) = TR(ba),$$ thus one can take a product of
algebra elements on a zero-rotation loop up to cyclic order of the product. In situations where we
choose a representation of the algebra or in the case of finite dimensional Hopf algebras where one
can use right integrals \cite{KRH}, there are ways to make actual evaluations of such traces. Here we
use them formally to indicate the result of concentrating the algebra on the loop.
\smallbreak

\begin{figure}
     \begin{center}
     \begin{tabular}{c}
     \includegraphics[width=8cm]{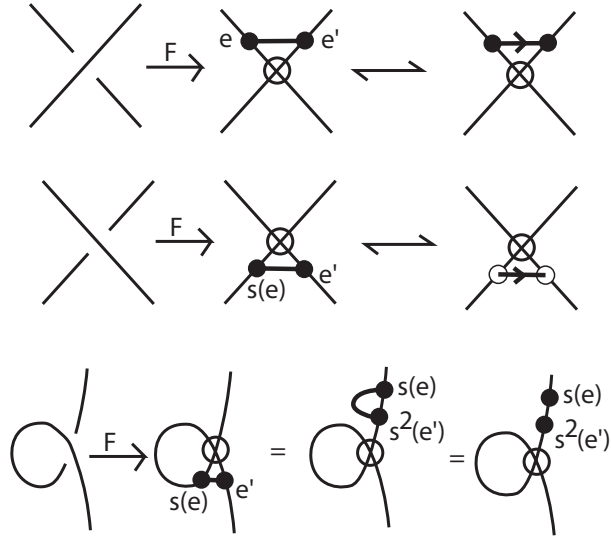}
     \end{tabular}
     \caption{The functor $F: VTC \longrightarrow Cat(A)$}
     \label{Functor}
\end{center}
\end{figure}

\begin{figure}
     \begin{center}
     \begin{tabular}{c}
     \includegraphics[width=8cm]{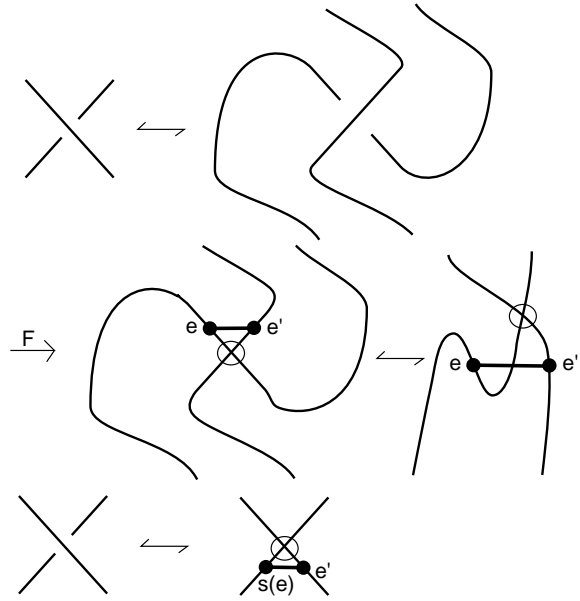}
     \end{tabular}
     \caption{Inverse and antipode}
     \label{twist}
\end{center}
\end{figure}

One further comment is in order about the antipode. In Figure~\ref{twist} we show that our axiomatic assumption about the antipode (the sliding rule around maxima and minima) actually demands that
the inverse of $\rho$ is $(s \otimes 1_{A})\circ \rho = (1_{A} \otimes s )\circ \rho$. This follows by
examining the form of the inverse of the positive crossing in the tangle category by turning that crossing
to produce an identity between the positive crossing and the negative crossing twisted with additional
maxima and minima. This relationship shows that if we set the functor $F$ on a right-handed crossing as we have done, then the way it maps the inverse crossing is forced and that this inverse corresponds to the inverse of $\rho$ in the quantum algebra. Thus the quantum algebra formula for the inverse of
$\rho$ is forced by the topology.
\smallbreak

In Figure~\ref{ApplyFunctor} we illustrate the entire functorial process for  the virtual trefoil of Figure~\ref{tangprod}. The virtual trefoil is denoted by $K$, and we find that $F(K)$ reduces to a zero-rotation circle with the
inscription $ e' s(f) s^{2}(e) s^{3}(f') G^{2}. $ We can, therefore, write the equation
$$F(K) = TR[e' s(f) s^{2}(e) s^{3}(f') G^{2}].$$
Another way to think about this trace expression is to regard it as a Gauss code for the knot that has extra structure. {\it The chords in the Gauss diagram are the string connectors of the beginning of this paper, generalized to the algebra category $Cat(A).$} The powers of the antipode and the power of $G$ keep track of rotational features in the diagram as it lives in the tangle category up to regular isotopy.
We now see that the mapping of the virtual braid group to the braid group generated by permutations
and string connectors has been generalized to the functor $F$ taking the virtual tangle category
to the abstract category of a quantum algebra. We regard this generalization as an appropriate context for thinking about virtual knots, links and braids.
\smallbreak

The category $Cat(A)$ of a quantum algebra $A$ can be generalized to an abstract category
with labels, virtual crossings, and with stringy connections that satisfy the algebraic Yang-Baxter equation. Each such stringy connection has a left label $e$ or $s(e)$ and a right label $e'.$ We retain the formalism of the antipode
as a formal replacement for adjoining a label with a cup and a cap. The resulting {\it abstract algebra category} will be denoted by $\overline{Cat(A)}$. Since we take this category with no further relations,
the functor $ \overline{F}:VTC \longrightarrow \overline{Cat(A)}$ is an equivalence of categories. This functor is
the direct analog of our reformulation of the virtual braid group in terms of stringy connectors.

\begin{figure}
     \begin{center}
     \begin{tabular}{c}
     \includegraphics[width=9cm]{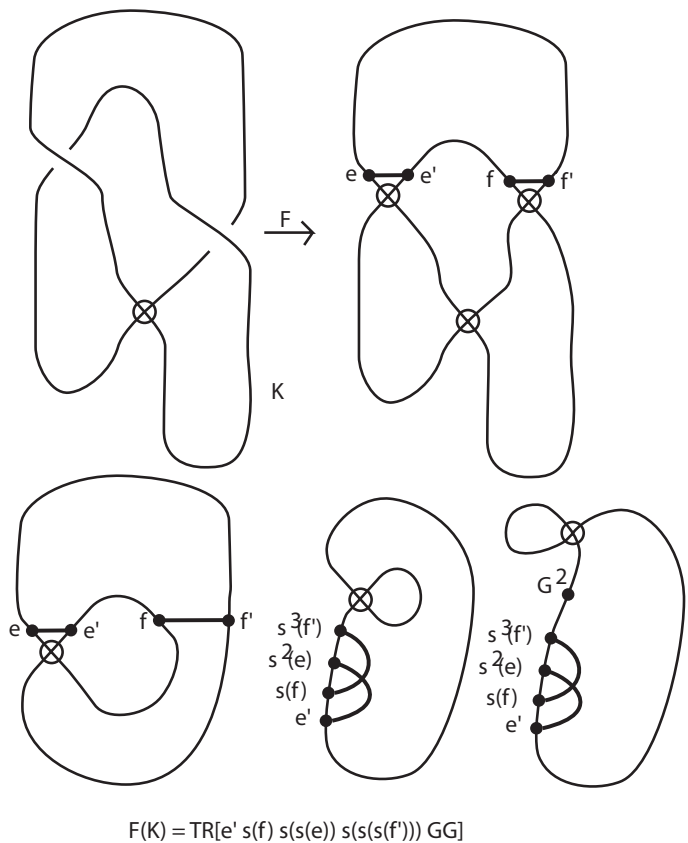}
     \end{tabular}
     \caption{The functor $F: T \longrightarrow Cat(A)$ applied to a virtual trefoil}
     \label{ApplyFunctor}
\end{center}
\end{figure}

\begin{figure}
     \begin{center}
     \begin{tabular}{c}
     \includegraphics[width=6cm]{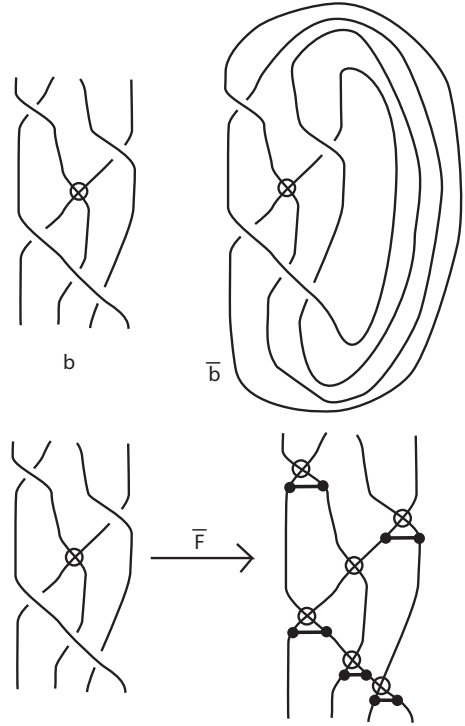}
     \end{tabular}
     \caption{Virtual braid and closure}
     \label{vbclosure}
\end{center}
\end{figure}

\subsection{Virtual Braids and Their Closures}
The functor $F:VTC \longrightarrow Cat(A)$ defined in the last subsection can be restricted to
the category of virtual (unoriented) braids that we will denote here by $VB.$ If the reader then examines
the result of this functor he will see that the image of a virtual crossing is a virtual crossing, and the image of a braid generator is a string connection (expressed in terms of $Cat(A).$). If we use the
corresponding functor
$$ \overline{F}:VB \hookrightarrow VTC \longrightarrow \overline{Cat(A)},$$ then the image
$ \overline{F}(VB)$ is an abstract category of
string connections and permutations that is (up to orientation) identical with our String Category $SC$
studied throughout this paper. This remark brings us full circle and shows how the String Category
fits in the context of the quantum link invariants discussed in this part of the paper. In particular,
view the bottom part of Figure~\ref{vbclosure} where we have illustrated the image under
$ \overline{F}$ of a particular virtual braid. Each classical crossing in the braid is replaced by a string connector followed by a virtual crossing. The string connector is interpreted as in the abstract Hopf algebra category, but in the braid image there is no other structure than the connectors and the virtual crossings. This shows how the braid lands in a subcategory that is isomorphic with our main category
$SC.$
\smallbreak

Now recall that one can move from virtual braids to virtual knots and links by taking the {\it braid closure.}
The closure $\overline{b}$ of a braid $b$ is obtained by attaching planar disjoint arcs from the outputs of a braid to its inputs as illustrated in Figure~\ref{vbclosure}. The result of the closure is a virtual knot or link. In particular, this means that we can express rotational quantum link invariants by applying
$F$ to the closure of virtual braid and then taking the trace $TR$ described in the last section.
alternatively, one can regard the invariant as $A$-valued where $A$ is the quantum algebra that supports the functor $F.$ Altogether, this subsection and the examples in Figure~\ref{vbclosure}
indicate the close relationship of the different constructions that have been outlined in this paper and
how the structure of the virtual braid group is intimately related to quantum link invariants for rotational
virtual links.

\subsection{Hopf Algebras and Kirby Calculus}
In Figure~\ref{kirby} we illustrate how one can use this concentration of algebra on the loop in the context of a Hopf algebra that has a right integral. The right integral is a function
$\lambda: A \longrightarrow k$ satisfying $$\lambda(x) 1_{A} = \Sigma \lambda(x_{1})x_{2}$$
where the coproduct in the Hopf algebra has formula $\Delta(x) = \Sigma x_{1} \otimes x_{2}$.
Here we point out how the use of the coproduct corresponds to doubling the lines in the diagram, and that if one were to associate the function
$\lambda$ with a circle with rotation number one,  then the resulting link evaluation will be invariant under the so-called Kirby move.
The Kirby move replaces two link components with new ones by doubling one component and connecting one of the components of the double with the other component. Under our functor from the virtual tangle category to the category for the Hopf algebra, a knot goes to a circle with algebra concentrated at $x.$ The doubling of the knot goes to concentric circles labeled with the coproduct
$\Delta(x) = \Sigma x_{1} \otimes x_{2}.$  Figure~\ref{kirby} shows how invariance under the handle-slide in the tangle category corresponds the integral equation
$$\lambda(x) y = \Sigma \lambda(x_{1})x_{2} y.$$
It turns out that
classical framed links $L$ have an associated compact oriented three manifold $M(L)$ and that
two links related by Kirby moves have homeomorphic three-manifolds. Thus the evaluation of links using the right integral yields invariants of three-manifolds. Generalizations to virtual three-manifolds are under investigation \cite{DK1}. We only sketch this point of view here, and refer the reader to \cite{KRH}.
\smallbreak

\begin{figure}
     \begin{center}
     \begin{tabular}{c}
     \includegraphics[width=8cm]{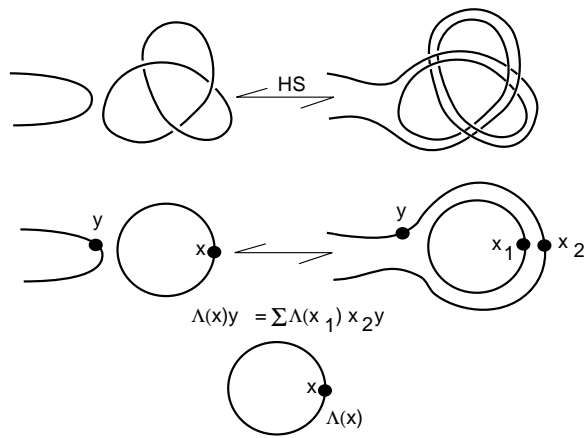}
     \end{tabular}
     \caption{The Kirby move}
     \label{kirby}
\end{center}
\end{figure}

\subsection {Hopf Algebra}

This section is added for reference about Hopf algebras. Quasitriangular Hopf algebras are an
important special case of the quantum algebras discussed in this section.
\smallbreak

Recall that a {\it Hopf algebra} \cite{Sweedler} is a bialgebra  $A$ over a commutative ring $k$ that has an associative multiplication $m:A \otimes A \longrightarrow A,$
and a coassociative comultiplication, and is equipped with a counit, a unit
and an antipode. The ring $k$ is usually taken to be a field.  The associative law for the multiplication  $m$ is expressed by the equation $$m(m \otimes1_{A}) = m(1_{A} \otimes m)$$ where $1_{A}$ denotes the identity map on A.
\smallbreak

The coproduct
$\Delta :A \longrightarrow A \otimes A$ is an algebra homomorphism and is
coassociative in the sense that
$$(\Delta \otimes 1_{A})\Delta = (1_{A} \otimes \Delta) \Delta.$$

The {\it unit} is a mapping from $k$ to $A$ taking $1_k$ in $k$ to $1_A$ in $A$ and, thereby, defining an action of $k$ on $A.$
It will be convenient to just identify the $1_k$ in $k$ and the $1_A$ in $A$, and to ignore
the name of the map that gives the unit.
\smallbreak

The counit is an algebra mapping from $A$ to $k$ denoted by $\epsilon :A \longrightarrow k.$ The following formula for the counit dualize the structure inherent in the unit:
$$(\epsilon \otimes 1_{A}) \Delta = 1_{A} = (1_{A} \otimes \epsilon) \Delta.$$

It is convenient to write formally
$$\Delta (x) = \sum x_{1} \otimes x_{2} \in A \otimes A$$
to indicate the decomposition of the coproduct of $x$ into a sum of first and second factors in the two-fold tensor product of $A$ with itself. We shall often drop the summation sign and write
$$\Delta (x) = x_{1} \otimes x_{2}.$$

The antipode is a mapping $s:A \longrightarrow A$ satisfying the equations
$$m(1_{A} \otimes s) \Delta (x) = \epsilon (x)1_{A} \quad \mbox{and} \quad m(s \otimes 1_{A}) \Delta (x)= \epsilon (x)1_{A}.$$
It is a consequence of this definition that $s(xy) = s(y)s(x)$ for all $x$ and $y$ in A. \vspace{3mm}

A {\it quasitriangular Hopf algebra} \cite{Drinfeld} is a Hopf algebra  $A$ with an element
$\rho \in A \otimes A$ satisfying
the following conditions:
\vspace{3mm}

\noindent
(1) $\rho \Delta = \Delta' \rho$ where $\Delta'$ is the composition
of $\Delta$ with the map on
$A \otimes A$ that switches the two factors. \vspace{3mm}

\noindent
(2) $$\rho_{13} \rho_{12} = (1_{A} \otimes \Delta) \rho,$$ $$\rho_{13} \rho_{23} = (\Delta \otimes 1_{A})\rho.$$

The symbol $\rho_{ij}$ denotes the placement of the first and second tensor factors of $\rho$ in the $i$ and $j$ places in a triple tensor product. For example, if $\rho = \sum e \otimes e'$ then $$\rho_{13} = \sum e \otimes 1_{A} \otimes e'.$$

Conditions (1) and (2) above imply that $\rho$ has an inverse and that

$$ \rho^{-1} = (1_{A} \otimes s^{-1}) \rho = (s \otimes 1_{A}) \rho.$$

It follows easily from the axioms of the quasitriangular Hopf algebra that $\rho$ satisfies the Yang-Baxter equation

$$\rho_{12} \rho_{13} \rho_{23} = \rho_{23} \rho_{13} \rho_{12}.$$

A less obvious fact about quasitriangular Hopf algebras is that there exists an element $u$ such that $u$ is invertible and $s^{2}(x) = uxu^{-1}$ for all $x$ in $A.$ In fact, we may take $u = \sum s(e')e$ where $\rho = \sum e \otimes e'.$ This result, originally due to Drinfeld \cite{Drinfeld}, follows from the diagrammatic categorical context of this paper. \vspace{3mm}

An element $G$ in a Hopf algebra is said to be {\em grouplike} if $\Delta (G) = G \otimes G$ and $\epsilon (G)=1$ (from which it follows that $G$ is invertible and $s(G) = G^{-1}$). A quasitriangular Hopf algebra is said to be a {\em ribbon Hopf algebra} \cite{RTG,KRH} if there exists a grouplike element $G$ such that (with $u$ as in the previous paragraph) $v = G^{-1}u$ is in
the center of $A$ and $s(u) = G^{-1}uG^{-1}$. We call $G$ a special grouplike element of $A.$
\vspace{3mm}

Since $v=G^{-1}u$ is central, $vx=xv$ for all $x$ in $A.$ Therefore $G^{-1}ux = xG^{-1}u.$ We know that $s^{2}(x) = uxu^{-1}.$ Thus $s^{2}(x) =GxG^{-1}$ for all $x$ in $A.$ Similarly, $s(v) = s(G^{-1}u) = s(u)s(G^{-1})=G^{-1}uG^{-1}G =G^{-1}u=v.$ Thus, the square of the
antipode is represented as conjugation by the special grouplike element in a ribbon Hopf algebra, and the central element $v=G^{-1}u$ is invariant under the antipode.
\smallbreak

This completes the summary of Hopf algebra properties that are relevant to the last section of the paper.




\begin{thebibliography}{AF}

\bibitem{Bardakov} V. G. Bardakov,
The virtual and universal braids, {\it Fund. Math.} {\bf 184} (2004), 1--18.

\bibitem{BER}
L. Bartholdi, B. Enriquez, P. Etingof and E. Rains, Groups and Lie algebras corresponding to the Yang-Baxter equations, {\em J. Alg.} {\bf 305}  (2006), 742--764.

\bibitem{Bellingeri} V. G. Bardakov and P. Bellingeri, Combinatorial properties of virtual braids,
{\em Topology and Its Applications}  {\bf 156} , No. 6 (2009), 1071--1082.

\bibitem{CS1}
J.S.Carter, S. Kamada and M. Saito, Stable equivalence of knots on surfaces and virtual  knot
cobordisms, in ``Knots 2000 Korea, Vol. 1 (Yongpyong)", {\em JKTR} {\bf 11}, No. 3 (2002), 311--320.

\bibitem{Drinfeld}
V.G. Drinfeld, Quantum groups,
{\em Proceedings of the International
Congress of Mathematicians, Berkeley, California, USA}(1987), 798--820.

\bibitem{DK}
H. Dye and L. H. Kauffman,  Minimal surface representations of virtual knots and links.
{\it Algebr. Geom. Topol.}  {\bf 5} (2005), 509--535.

\bibitem{DK1} H. Dye and L. H. Kauffman, Virtual knot diagrams and the
Witten-Reshetikhin-Turaev invariant. math.GT/0407407, {\it J. Knot Theory Ramifications} 14 (2005), no. 8, 1045--1075.

\bibitem{FRR}
R. Fenn, R. Rimanyi, C. Rourke, The braid permutation group, {\em Topology}  {\bf 36} (1997), 123--135.

\bibitem{Fox} R. Fox, A quick trip through knot theory, in ``Topology of Three-Manifolds" ed. by M.K.Fort, Prentice-Hall Pub. (1962).

\bibitem{Paris}
E. Godelle and L. Paris,
$K(\pi, 1)$ and word problems for infinite type Artin-Tits groups, and applications to virtual
braid groups, arXiv:1007.1365 v1, 8 Jul 2010.


\bibitem{GPV}
 M. Goussarov, M.Polyak and O. Viro, Finite type invariants of classical and virtual knots,
{\em Topology}  {\bf 39} (2000),  1045--1068.

\bibitem{HR}
D. Hrencecin, ``On Filamentations and Virtual Knot Invariants" Ph.D Thesis, Unviversity
of Illinois at Chicago (2001).

\bibitem{HRK} D. Hrencecin and L. H. Kauffman, ``On Filamentations and Virtual Knots",  {\em Topology
and Its Applications}  {\bf 134} (2003), 23--52.

\bibitem{Jones}
 Jones, V. F. R. Hecke algebra representations of braid groups and link polynomials. {\em Ann. of Math.} {\bf 126}, No. 2,  (1987), No. 2, 335--388.

\bibitem{KADOKAMI}
T. Kadokami, Detecting non-triviality of virtual links,{\it   JKTR} {\bf 6}, No. 2  (2003), 781--803.

\bibitem{Kamada} S. Kamada, Braid presentation of virtual knots and welded knots, preprint March 2000.

\bibitem{KRH}
L.H. Kauffman and D.E. Radford, Invariants of 3- manifolds derived from finite dimensional Hopf algebras, {\em JKTR} {\bf 4}, No. 1 (1995), 131--162.

\bibitem{VKT}
L. H. Kauffman, Virtual Knot Theory , {\em European J. Comb.}  {\bf 20} (1999), 663--690.

\bibitem{SVKT} L. H. Kauffman, A Survey of Virtual Knot Theory, {\em Proceedings of Knots  in
Hellas~'98}, World Sci. 2000, 143--202.

\bibitem{DVK} L. H. Kauffman, Detecting Virtual Knots, {\em Atti. Sem. Mat. Fis. Univ. Modena
Supplemento al  Vol. IL } (2001),  241--282.

\bibitem{KL3} L. H. Kauffman, S. Lambropoulou, Virtual Braids, {\em Fund. Math.} {\bf 184} (2004), 159--186.

\bibitem{KL4}  L.H. Kauffman, S. Lambropoulou, Virtual braids and the L-move, {\em JKTR} {\bf 15}, No. 6 (2006), 773--811.

\bibitem{KP} L. H. Kauffman,
``Knots and Physics",  World Sci.  (1991), Second Edition 1994 (723 pages), Third Edition 2001.

\bibitem{KCalc} R. Kirby, A calculus for framed links in $S^{3}$, {\em Invent.
Math.}, {\bf 45} (1978), 35--56.


\bibitem{KiSa} T. Kishino and S. Satoh, A note on non-classical virtual knots, {\em JKTR} {\bf 13}, No. 7  (2004), 845--856.

\bibitem{KUP}
G. Kuperberg, What is a virtual link?, {\em Algebraic and Geometric Topology}  {\bf 3}
2003, 587--591.

\bibitem{MKS} W. Magnus, A. Karass, D. Solitar, ``Combinatorial Group Theory", Dover Pub., N. Y. (1966, 1976).

\bibitem{M} V. Manturov, ``O raspoznavanii virtual'nykh kos" (On the
Recognition of Virtual Braids), Zapiski Nauchnykh Seminarov POMI
(Petersburg Branch of Russ. Acad. Sci. Inst. Seminar Notes),
299. {\em Geometry and Topology} {\bf 8} (2003),  267--286.

\bibitem{Ohtsuki} T. Ohtsuki, ``Quantum Invariants" , Vol. 29 - WS Series on Knots and Everything, World Scientific Pub. (2002).

\bibitem{RTG}
N. Yu. Reshetikhin and V.G. Turaev,
Ribbon graphs and their invariants derived from quantum groups.
{\em Commun. Math. Phys. }  {\bf 127} (1990), 1--26.

\bibitem{RT} N.Y. Reshetikhin and V. Turaev, Invariants of Three-Manifolds via
link polynomials and quantum groups, {\em Invent. Math.}  {\bf 103} (1991),
547--597.

\bibitem{Satoh} S. Satoh,  Virtual knot presentation of
ribbon torus-knots, {\em JKTR}  {\bf 9} (2000), No.4,  531--542.

\bibitem{Stillwell} J. Stillwell, ``Classical Topology and Combinatorial Group Theory", Springer-Verlag (1980,1993).

\bibitem{Sweedler} M.E. Sweedler, ``Hopf Algebras", Mathematics Lecture Notes Series,
Benjamin, New York, 1969.

\bibitem{TURAEV}
V. Turaev, Virtual strings and their cobordisms,
arXiv.math.GT/0311185.

\bibitem{V1}
V. V. Vershinin, On the homology of virtual braids and the Burau representation, {\em JKTR} 10, No. 5, 795--812.

\bibitem{V2}
V. V. Vershinin, On the singular braid monoid, {\em St. Petersburg Math. J.}  {\bf 21}, No. 5, 693--704.

\end{thebibliography}
\end{document}